\documentclass{amsart}

\usepackage{amsmath}
\usepackage{amssymb}
\usepackage{stmaryrd}
\usepackage[backgroundcolor = white, linecolor = red, bordercolor =red]{todonotes}
\usepackage[all,cmtip]{xy}
\usepackage{tikz}
\usepackage{xcolor}
\usepackage[hidelinks]{hyperref}

\theoremstyle{plain}
\newtheorem{theorem}{Theorem}
\newtheorem{corollary}[theorem]{Corollary}
\newtheorem{lemma}[theorem]{Lemma}
\newtheorem{proposition}[theorem]{Proposition}

\theoremstyle{remark}
\newtheorem{remark}[theorem]{Remark}

\numberwithin{equation}{section}
\numberwithin{theorem}{section}
\newcommand{\R}{\mathbb{R}}
\newcommand{\N}{\mathbb{N}}
\newcommand{\Domain}{\Omega}
\newcommand{\Dim}{d}
\newcommand{\Normal}{\mathsf{n}}
\newcommand{\Poly}{P}
\newcommand{\dt}{}
\newcommand{\Grad}{\nabla}
\newcommand{\SymGrad}{\Grad_S}
\newcommand{\Div}{\mathrm{div}}
\newcommand{\Norm}[1]{\| #1 \|}
\newcommand{\Normtr}[1]{ \lvert\!\lvert\!\lvert{#1} \rvert\!\rvert\!\rvert}
\newcommand{\bbU}{\mathbb{U}}
\newcommand{\bbP}{{\mathbb{P}}}
\newcommand{\bbD}{\mathbb{D}}
\newcommand{\bbY}{\mathbb{Y}}
\newcommand{\bbX}{\mathbb{X}}
\newcommand{\bbE}{\mathbb{E}}
\newcommand{\bbL}{\mathbb{L}}
\newcommand{\calE}{\mathcal{E}}
\newcommand{\calD}{\mathcal{D}}
\newcommand{\calL}{\mathcal{L}}
\newcommand{\calP}{\mathcal{P}}
\newcommand{\calQ}{\mathcal{Q}}
\newcommand{\Tot}{\text{tot}}
\begin{document}

% Title
\title[Inf-sup theory for the quasi-static Biot's equations]{Inf-sup theory for the quasi-static\\ Biot's equations in poroelasticity}

% Authors
\author[C.~Kreuzer]{Christian Kreuzer}
\address{TU Dortmund \\ Fakult{\"a}t f{\"u}r Mathematik \\ D-44221 Dortmund \\ Germany}
\email{christian.kreuzer@tu-dortmund.de}

\author[P.~Zanotti]{Pietro Zanotti}
\address{Universit\`{a} degli Studi di Milano, Dipartimento di Matematica, 20131, Milano, Italy}
\email{pietro.zanotti@unimi.it}

% Keywords

% Subject classification

% 35Q35: PDEs in connection with fluid mechanics
%
% 35Q74: PDEs in connection with mechanics of deformable solids
%
% 76S05: Flows in porous media; filtration; seepage
%
% 74F10: Fluid-solid interactions (including aero- and hydro-elasticity, porosity, etc.)
%
% 74H20: Existence of solutions of dynamical problems in solid mechanics
%
% 74H25: Uniqueness of solutions of dynamical problems in solid mechanics
%
% 74H55: Stability of dynamical problems in solid mechanics

\begin{abstract}
We analyze the two-field formulation of the quasi-static Biot's
equations in bounded domains by means of the inf-sup theory. For this purpose, we exploit
an equivalent four-field formulation of the equations, introducing the
so-called total pressure and total fluid content as independent
variables. We establish existence, uniqueness and stability of the
solution. Our stability estimate is two-sided and robust, meaning that
the regularity established for the solution matches the regularity requirements for the data and
the involved constants are independent of all material parameters. We
prove also that additional regularity in space of the data implies, in
some cases, corresponding additional regularity in space of the
solution. These results are instrumental to the design and the
analysis of discretizations enjoying accurate stability and error
estimates.  
\end{abstract} 

\maketitle

\section{Introduction}
\label{S:introduction}

The analysis and the discretization of the quasi-static Biot's equations have been the subject of several studies in recent years. The equations arise in the
theory of poroelasticity and model the flow of a fluid inside an
elastic medium. They fit into the abstract framework
\begin{equation}
\label{E:linear-equations}
\mathcal{B} y = \ell
\end{equation}
for a linear operator $\mathcal{B}$, with $y$ and $\ell$ denoting the
solution and the load, respectively; see \eqref{E:BiotProblem} for the
specific definitions.  

In this paper, we develop some analytical results to be used in a follow-up paper \cite{Kreuzer.Zanotti:22+} regarding the discretization of the Biot's equations. More precisely, for bounded domains, we establish existence, uniqueness, two-sided stability and regularity in space of the solution by means of the inf-sup theory, i.e., via
the so-called Banach-Ne\u{c}as-Babu\v{s}ka theorem,
cf. \cite{Necas:62}. This approach and various aspects of our results
are new to our best knowledge. The main advantage stemming from the
use of the inf-sup theory is that we obtain a two-sided stability
estimate in the form 
\begin{equation}
\label{E:linear-equations-stability}
\Norm{y}_1 \eqsim \Norm{\ell}_{2,*}.
\end{equation}
In other words, differently from previous works, we are able
to prove that the operator $\mathcal{B}$ in \eqref{E:linear-equations}
establishes an isomorphism between the space for the solution and the one
for the load. Hence, the regularity established for the solution matches the regularity requirements for the data. In passing, we relax also the assumption made on the regularity of the load in previous references.

A major challenge in our analysis
consists in selecting the norms in
\eqref{E:linear-equations-stability}, i.e., in characterizing the
regularity of the solution and of the load in the equations. The
difficulty hinges on the action of the differential operator
$\mathcal{B}$ in \eqref{E:linear-equations}, which couples the two
components of the solution (the displacement of the elastic medium and
the pressure of the fluid) in a highly nontrivial way,
cf. \eqref{E:BiotProblem-equations} below. We deal with this
issue by considering an equivalent four-field formulation of the
Biot's equations from \cite{Khan.Zanotti:22}, that is obtained by
introducing the so-called total pressure and total fluid content as
independent variables. The new
formulation motivates the definition of the norm $\Norm{\cdot}_{2,*}$
as a product norm. The definition of the norm $\Norm{\cdot}_1$ is more
involved, as it still couples the regularity of two components of the
solution. This
coupling of the regularity is indeed necessary.

Once the norms are selected, we establish the above-mentioned results
by the inf-sup theory. This technique, differently from other ones,
always implies a two-sided stability estimate, like
\eqref{E:linear-equations-stability}, when it can be applied. For this
purpose, one has to verify that the bilinear form induced by the
operator $\mathcal{B}$ in \eqref{E:linear-equations} fulfills three
properties: boundedness, inf-sup stability and nondegeneracy. The use
of the inf-sup theory was made popular in numerical analysis by the
pioneering works of Babu\v{s}ka \cite{Babuska:70} and Brezzi
\cite{Brezzi:74}, who proposed the technique to obtain accurate
stability and a priori error estimates for a linear equation and its
discretization. Later on inf-sup stability has been noticed to be
important also for a posteriori error estimation \cite{Verfuerth:13}
and for the convergence of adaptive discretizations
\cite{Feischl:22,Morin.Siebert.Veeser:08}. Still, for some reason, the
use of the inf-sup theory has been mostly confined to stationary
equations and only recently there have been attempts to apply it to
evolutionary ones, see
\cite{Ern.Guermond:21c,Steinbach.Zank:22,Tantardini.Veeser:16}.  

This paper aims at further contributing to the development of the
infs-sup theory for evolution equations, developing tools to be later
used for the numerical analysis in \cite{Kreuzer.Zanotti:22+} as
well. We intend also to highlight the benefits of our technique in
comparison with other ones. In addition to the early contribution
\cite{Auriault.SanchezPalencia:77} (restricted to a rather specific
case), we are aware of two other approaches to the analysis of the
Biot's equations.  

\v{Z}en\'{\i}\v{s}ek \cite{Zenisek:84} used the
so-called Faedo-Galerkin (Rothe's) method, in combination with a ingenious way of
testing the equations, in order to infer stability. This technique is
quite popular also in numerical analysis for the derivation of error
estimates, see e.g. \cite{Phillips.Wheeler:07a}. Other results in this
flavor can be found in
\cite{Botti.Botti.DiPietro:21,Li.Zikatanov:22,Owczarek:10}.  Another
possible technique is the one of Showalter \cite{Showalter:00}, who
studied both strong and weak solutions by the theory of implicit
evolution 
equations. Both approaches assume more regular data than our one and
do not establish an equivalence like
\eqref{E:linear-equations-stability}. Extensions to nonlinear problems
in poromechanics are found, e.g., in
\cite{Bociu.Guidoboni.Sacco.Webster:16,Bociu.Muha.Webster:22,Bociu.Webster:21,Cao.Chen.Meir:13,Cao.Chen.Meir:14}.  

Finally, motivated by the numerical analysis in
\cite{Kreuzer.Zanotti:22+}, we are interested also in shift theorems,
i.e., in determining if more regular data give rise to more regular
solutions. We give a positive answer for the regularity in space,
under a set of relatively restrictive assumptions, by using again the
inf-sup theory. The relaxation of our assumptions and the regularity
in time are more challenging tasks and we do not discuss them
here. For instance, it is known that the time derivative of the
solution can be singular at the initial time,
cf. \cite[Section~2]{Murad.Thomee.Loula:96} and
\cite[Section~3]{Showalter:00}. We are aware of only few other
regularity results, see \cite{Botti.Botti.DiPietro:21,Yi:17}. 

\subsection*{Contribution} Summarizing, we propose a new approach to
the analysis of the Biot's equations in bounded domains, which is used to establish both
well-posedness and additional regularity in space. Differently from
previous contributions, we make a weaker regularity assumption on the
data and we make sure that it matches with the regularity guaranteed
for the solution by establishing two-sided stability estimates like
\eqref{E:linear-equations-stability}. All constants in our bounds are
robust with respect to the time horizon and the material parameters in the equations. In
particular, we treat at the same time the critical case of vanishing
and nonvanishing specific storage coefficient, also in combination
with general boundary conditions. 

\subsection*{Organization} In Section~\ref{S:BiotEquations} we recall
the Biot's equations and introduce the setting for our analysis. In
Section~\ref{S:WellPosedness} we establish existence, uniqueness and
two-sided stability of the solution. In
Section~\ref{S:Stability-Additional} we investigate the stability
estimate more extensively. In Section~\ref{S:Regularity} we establish
additional regularity in space. 

\subsection*{Notation} We denote by $L^2(\bbX)$, $H^1(\bbX)$ and $C^0(\bbX)$ the
spaces of all $L^2$, $H^1$ and $C^0$ functions mapping the time
interval $[0, T]$ into a Banach space $\bbX$, equipped with the norm
$\Norm{\cdot}_\bbX$. The symbol $\left\langle \cdot, \cdot
\right\rangle_\bbX$ indicates the duality of $\bbX$ and $\bbX^*$. For
$\bbX = L^2(\Domain)$, $\Domain \subseteq \R^\Dim$, we use the
abbreviations $\Norm{\cdot}_\Domain$ for the norm and $(\cdot,
\cdot)_\Domain$ for the corresponding scalar product.  We write $a
\lesssim b$ and 
$a \eqsim b$ when there are constants $0 < \underline{c} \leq
\overline{c}$ such that $a \leq \overline{c}\,b$ and $\underline{c}
\,b \leq a \leq \overline{c}\,b$, respectively. As a rule of thumb,
the hidden constants are independent of the time horizon and of the material parameters
involved in the equations. The dependence on other relevant quantities
is addressed case by case.

\section{Biot's equations and abstract weak formulation}
\label{S:BiotEquations}

In this section we introduce the Biot's equations and propose a framework for analyzing them by the inf-sup theory.

\subsection{Two-field formulation}
\label{SS:two-field}

Let $\Domain \subseteq \R^\Dim$, $\Dim \in \N$, be a bounded domain (i.e. a bounded, open and connected set) whose boundary can be locally represented as the graph of a
Lipschitz-continuous function. The quasi-static Biot's equations in
$\Domain \times (0,T)$, $T > 0$, read as 
\begin{subequations}
\label{E:BiotProblem}
\begin{equation}
\label{E:BiotProblem-equations}
\begin{alignedat}{2}
 -\Div( 2 \mu \SymGrad u + (\lambda \Div u  - \alpha p)I) &= f_u \qquad &\text{in $\Domain \times (0,T)$  }\\
\partial_t (\alpha \Div u + \sigma p) - \Div(\kappa \Grad p) &= f_p \qquad  &\text{in $\Domain \times (0,T)$.}
\end{alignedat}
\end{equation}
The equations model the flow of a Newtonian fluid inside a linear
elastic porous medium. The first one states the momentum balance, the
second one states the mass balance. The unknowns are the displacement
$u: \Domain \times (0,T) \to \R^\Dim$ of the medium and the pressure
$p: \Domain \times (0,T) \to \R$ of the fluid. The symbols $\SymGrad$
and $I$ denote the symmetric part of the gradient and the $\Dim \times
\Dim$ identity tensor respectively. Moreover, the following material parameters are
involved: the Lam\'{e} constants $\mu, \lambda > 0$, the Biot-Willis
constant $\alpha~>~0$, the constrained specific storage coefficient
$\sigma \geq 0$ and the hydraulic conductivity $\kappa > 0$. For
simplicity, we assume that 
\begin{center}
	all parameters are constant in $\Domain \times (0,T)$.
\end{center}

\begin{remark}[Parameters]
\label{R:parameters}
Our subsequent analysis could be applied, up to minor modifications,
under the assumption that some of the material parameters are functions of space and/or time with uniform upper and lower bound. Of course, the ratio of the upper and the lower bound
would affect the constants in our estimates. A restriction in our approach is that $\alpha$ should be constant in space, because we sometimes
commute it with space derivatives. Similarly, $\kappa$ should be
constant in time. Finally, although we can
treat also the case $\sigma = 0$, 
it is unclear to us whether this parameter can be zero and nonzero in
different regions of $\Domain \times (0, T)$. The assumption on $\kappa$ is especially critical, because it prevents from a direct extension of our analysis to nonlinear poroelastic models, where $\kappa$ may depend on the solid dilation $\Div u$, see e.g.
\cite[Section~2.3]{Bociu.Muha.Webster:22}.
\end{remark}

We complement the Biot's equations \eqref{E:BiotProblem-equations} by the initial condition 
\begin{equation}
\label{E:BiotProblem-initialcondition}
(\alpha \Div u + \sigma p)_{|t=0} = \ell_0 \quad \text{in} \quad \Domain
\end{equation}
and by the boundary conditions 
\begin{equation}
\label{E:Biot-BCs-mixed}
\begin{alignedat}{2}
u &= 0
&\quad &\text{on} \quad \Gamma_{u,E} \times (0,T)\\
(2 \mu \SymGrad u + (\lambda \Div u - \alpha p) I)\Normal &= g_{u}&
\quad& \text{on} \quad \Gamma_{u,N} \times (0,T) \\
p &= 0
&\quad &\text{on} \quad \Gamma_{p,E} \times (0,T)\\
\kappa \Grad p \cdot \Normal &= g_{p}
&\quad &\text{on} \quad \Gamma_{p,N} \times (0,T)
\end{alignedat}
\end{equation}
where $\Gamma_{u,E} \cup \Gamma_{u,N} = \partial \Domain =
\Gamma_{p,E} \cup \Gamma_{p,N}$ and $\Gamma_{u,E} \cap \Gamma_{u,N} =
\emptyset = \Gamma_{p,E} \cap \Gamma_{p,N}$. The letter $\Normal$
denotes the outward unit normal vector on $\partial \Domain$. The
subscripts `$E$' and `$N$' indicate essential and natural boundary
conditions, respectively. Inhomogeneous essential boundary conditions
can be treated as usual, by modifying the data in the equations
\eqref{E:BiotProblem-equations}.  
\end{subequations} 

\begin{remark}[Initial condition]
\label{R:initial-condition}
The time derivative acts in \eqref{E:BiotProblem-equations} only on a combination of $u$ and $p$, namely $\alpha \Div u + \sigma p$, thus suggesting that only the initial value of such auxiliary variable should be prescribed as done, e.g., in \cite[sections~3-4]{Showalter:00}. Still, different formulations are sometimes encountered in numerical analysis. Phillips and Wheeler
\cite{Phillips.Wheeler:07a} suggest to set $p(0)$ equal to
the hydrostatic pressure. Then, they evaluate the first equation in
\eqref{E:BiotProblem-equations} at $t=0$ and solve a linear elasticity
problem for $u(0)$. Other authors, see
e.g. \cite{Murad.Loula:92}, assume $\sigma = 0$ and set $\Div
u(0) = 0$ (that is equivalent to
\eqref{E:BiotProblem-initialcondition} with $\ell_0  = 0$). Then, they
evaluate the first equation in \eqref{E:BiotProblem-equations} at
$t=0$ and solve a Stokes problem for $u(0)$ and $p(0)$. In
other references, the values of $u(0)$ and $p(0)$ are just prescribed,
see e.g. \cite{Li.Zikatanov:22}. In this case, the compatibility of the prescribed initial values with the other data can be problematic and give rise to irregular behaviors, see \cite{Verri.Guidoboni.Bociu.Sacco:18}. 
\end{remark}

\begin{remark}[Boundary conditions]
\label{R:boundary-conditions}
The boundary conditions considered in \cite{Showalter:00} are slightly more sophisticated than~\eqref{E:Biot-BCs-mixed} in that they allow for a coupling on the intersection of the nonessential parts of the boundary \(\Gamma_{u,N}\cap\Gamma_{p,N}\), which requires additional regularity of the data and compatibility conditions on the initial values. In all other references we are aware of (from both the analytical and the numerical side), the boundary conditions~\eqref{E:Biot-BCs-mixed} (see \cite{Phillips.Wheeler:07a,Zenisek:84}) or simplified versions thereof are considered. 
\end{remark}

\begin{remark}[Unbounded domains]
\label{R:unbounded-domains}
Assuming that $\Domain$ is bounded excludes some relevant test cases from our analysis. For instance, \cite[Section~5.2]{Murad.Loula:92} considers the Biot's equations on a two-dimensional infinite strip, i.e. a domain that is bounded in one direction and unbounded in the other one. The application of our analysis to such domains is obstructed by the combination of two key technical tools. On the one hand, the constant $c$ in the first part of \eqref{E:inf-sup-B} depends on the anisotropy of $\Domain$. On the other hand, we need a Poincaré inequality for \eqref{E:abstract-setting-spaces-ptot-mb} and this requires that $\Domain$ is bounded at least along one direction.
\end{remark}

\subsection{Four-field formulation}
\label{SS:four-field}

Our starting point for the analysis of the initial-boundary value problem~\eqref{E:BiotProblem} are the results established in \cite{Khan.Zanotti:22}. In that reference, the stationary equations obtained after a time semi-discretization of \eqref{E:BiotProblem-equations} are considered. The inf-sup theory developed in \cite[Section~2]{Khan.Zanotti:22} reveals that, in the stationary case, it is possible to control two auxiliary variables, in addition to the displacement $u$ and the pressure $p$, namely the total pressure
\begin{subequations}
\label{E:auxiliary-variable}
\begin{equation}
\label{E:total-pressure}
p_\Tot := \lambda \Div u - \alpha \calP_\bbD p
\end{equation}
and the total fluid content 
\begin{equation}
\label{E:total-fluid-content}
m := \calP_{\overline \bbP}(\alpha \Div u) + \sigma p.
\end{equation} 
\end{subequations}
The operators $\calP_\bbD$ and $\calP_{\overline \bbP}$ are $L^2(\Domain)$-orthogonal projections, whose specific definition can be found in Section~\ref{SS:weak-formulation} below. Note that the definition of $p_\Tot$ is reminiscent of the  Herrmann mixed formulation of the linear elasticity equations, see \cite[Section~8.12.1]{Boffi.Brezzi.Fortin:13}.

Treating the total pressure and the total fluid content as independent unknowns leads to the following formulation of the Biot's equations~\eqref{E:BiotProblem-equations}
\begin{equation}
\label{E:BiotProblem-equations-fourfield}
\begin{alignedat}{2}
-\Div( 2 \mu \SymGrad u + p_\Tot I) &= f_u \qquad & \text{in $\Domain \times (0,T)  $ }\\
\lambda \Div u - p_\Tot - \alpha \calP_\bbD p &= 0 \qquad & \text{in $\Domain \times (0,T)  $ }\\
\alpha \calP_{\overline \bbP}\Div u + \sigma p - m &= 0 \qquad & \text{in $\Domain \times (0,T)  $ }\\
\partial_t m - \Div(\kappa \Grad p) &= f_p \qquad  & \text{in $\Domain \times (0,T).$}
\end{alignedat}
\end{equation}
Our use of the projections $\calP_\bbD$ and $\calP_{\overline \bbP}$ is motivated by our subsequent choice of the test functions for the equations. Similarly to \cite{Khan.Zanotti:22}, the analysis in the next sections equivalently applies to the original two-field formulation
\eqref{E:BiotProblem-equations} or to the above four-field formulation
of the Biot's equations. In our perspective, working with the latter one is slightly more convenient, cf. Remark~\ref{R:two-field-formulation} below.  

\begin{remark}[Auxiliary variables]
\label{R:AuxiliaryVariables}
Introducing auxiliary variables as independent unknowns is a common practice in numerical analysis, which can foster the construction of discretizations with specific stability and/or approximation properties. To our best knowledge, the idea of introducing the total pressure
$p_\Tot$ is relatively recent and dates back
to \cite{Lee.Mardal.Winther:17,Oyarzua.RuizBaier:16}. In contrast, we are not aware of any reference paying specific attention to the approximation of the total fluid content $m$, although this is a relevant variable, as pointed out by the formulation \eqref{E:BiotProblem-equations-fourfield}, the initial condition \eqref{E:BiotProblem-initialcondition} and also previous theoretical results like those in \cite[sections~3-4]{Showalter:00}.
\end{remark} 
 
\subsection{Weak formulation}
\label{SS:weak-formulation}

We propose the weak formulation of the equations
\eqref{E:BiotProblem-equations-fourfield}, with the initial and boundary conditions
\eqref{E:BiotProblem-initialcondition} and \eqref{E:Biot-BCs-mixed},
that is the subject of our analysis in the next sections. Some minor
differences are possible in the definition of the function spaces
involved in such formulation, depending on the boundary conditions and
on whether the constrained specific storage coefficient $\sigma$ vanishes or not. Therefore, we
introduce an abstract setting in order to treat all possible cases
simultaneously. 

The second-order elliptic operators involved in
\eqref{E:BiotProblem-equations-fourfield} and the boundary conditions
\eqref{E:Biot-BCs-mixed} suggest that $u$ and $p$ should be functions
with values in the spaces 
\begin{subequations}
\begin{align}
\label{E:abstract-setting-spaces-u-p}
\bbU &:= \begin{cases}
H^1(\Domain)^\Dim\cap \mathrm{RM}^\perp & \text{if $\Gamma_{u,N} = \partial \Domain$}\\
H^1_{\Gamma_{u,E}}(\Domain)^\Dim & \text{otherwise} 
\end{cases}\\
\intertext{and}
\label{E:abstract-setting-spaces-p}
\bbP &:= \begin{cases}
H^1(\Domain)\cap L^2_0(\Domain) & \text{if $\Gamma_{p,N} = \partial \Domain$}\\
H^1_{\Gamma_{p,E}}(\Domain) \cap L^2_0(\Domain) & \text{if $\Gamma_{p,N} \neq \partial \Domain, \Gamma_{u,E} = \partial \Domain, \sigma = 0$}\\
H^1_{\Gamma_{p,E}}(\Domain) & \text{otherwise.}
\end{cases} 
\end{align}
\end{subequations}
The definition of $\bbU$ and $\bbP$ in the pure Neumann case involves the $L^2$-orthogonal complement of rigid body motions $\mathrm{RM}$ in $L^2(\Domain)^\Dim$ and of constant functions in $L^2(\Domain)$. More precisely, we set $\mathrm{RM}^\perp := \{ q \in L^2(\Domain)^\Dim \mid (q,r)_\Domain = 0\; \forall r \in \mathrm{RM} \}$ and $L^2_0(\Domain) := \{ q \in L^2(\Domain) \mid \int_{\Domain} q = 0\}$.

\begin{remark}[Pressure space]
\label{R:pressure-space}
The second case in~\eqref{E:abstract-setting-spaces-p} is nonstandard, because it prescribes
  both the essential condition on a portion of the boundary and the zero
  integral mean in $\Domain$. In previous approaches, like e.g. \cite{Zenisek:84}, the boundedness of the pressure in the $L^2(H^1(\Domain))$-norm
  is established under a regularity assumption on the data. Thus, the $L^2(L^2(\Domain))$-norm can be controlled by a Poincar\'{e} inequality. In our approach, we assume less regular data, therefore we do not control the $L^2(H^1(\Domain))$-norm of the pressure, cf. Proposition~\ref{P:rough-trial-functions}. Hence, it appears for
  \(\Gamma_{u,E}=\partial\Omega\) and \(\sigma=0\), that we can
  only bound the $L^2(L^2(\Domain))$-norm of the pressure up to
  constant functions in space and thus the restriction to mean value free functions is needed. This informal guess is made more rigorous in Lemma~\ref{L:inf-sup}.
\end{remark}

Owing to \eqref{E:auxiliary-variable}, we regard $p_\Tot$ and $m$ as
functions with values in the following closed subspaces of $L^2(\Domain)$:
\begin{subequations}\label{E:abstract-setting-spaces-ptot-m}
  \begin{align}
    \label{E:abstract-setting-spaces-ptot-ma}
    \bbD := \Div(\bbU) &= \begin{cases}
      L^2_0(\Domain) &\text{if  $\Gamma_{u,E} = \partial \Domain$}\\
      L^2(\Domain) &\text{otherwise}
    \end{cases}
    \intertext{and}\label{E:abstract-setting-spaces-ptot-mb}
    \overline{\bbP} &= \begin{cases}
      L^2_0(\Domain) & \text{if
        $\Gamma_{p,N} = \partial \Domain$ or $\Gamma_{u,E} = \partial \Domain, \sigma = 0$}\\ 
      L^2(\Domain) & \text{otherwise.}
    \end{cases}
  \end{align}
\end{subequations}
The closure in the definition of $\overline{\bbP}$ is taken with respect to the $L^2(\Domain)$-norm. We denote by $\calP_\bbD$ and $\calP_{\overline \bbP}$ the $L^2$-orthogonal projections onto $\bbD$ and $\overline{\bbP}$, respectively.  

We equip the spaces $\bbU$
and $\bbP$ with $H^1(\Domain)$-like-norms scaled by $\sqrt{2\mu}$ and
$\sqrt{\kappa}$, respectively, and the spaces $\bbD$ and
\(\overline{\bbP}\) with the
$L^2(\Domain)$-norm. More precisely, we set
\begin{equation}
\label{E:norms}
\Norm{\cdot}_\bbU := \Norm{\sqrt{2\mu} \SymGrad\cdot}_{\Domain}
\qquad \text{and} \qquad
\Norm{\cdot}_\bbP := \Norm{\sqrt{\kappa} \Grad \cdot}_{\Domain}.
\end{equation}
Note that, by Korn's and Poincar\`e-Friedrichs inequalities, these are
indeed norms on their respective spaces. By standard functional analysis arguments, we have that
\begin{equation*}
\label{E:H0}
\bbP \subseteq \overline{\bbP} \quad  \text{is a dense compact subspace}.
\end{equation*}
We identify $\overline{\bbP}$ with $\overline{\bbP}^*$ via the
$L^2(\Domain)$-scalar product. Hence, $\bbP \subseteq \overline{\bbP}
\equiv \overline{\bbP}^* \subseteq\bbP^*$ is a Hilbert triplet and the
duality $\left\langle \cdot,\, \cdot \right\rangle_\bbP$ coincides with
$( \cdot,\, \cdot)_\Domain$ when both arguments are in
$\overline{\bbP}$.  

We introduce an abstract notation also
for the (weak form of) the differential operators involved in the
Biot's equations, so as to make many formulae more compact. We denote
by $\calE: \bbU \to \bbU^*$ and $\calL: \bbP \to \bbP^*$ the elliptic
operators acting on $u$ and $p$, respectively, and by $\calD: \bbU \to
\bbD$ the divergence, namely  
\begin{equation}
\label{E:abstract-operators-concrete}
\calE := - \Div (2\mu  \SymGrad) 
\qquad \qquad
\calD := \Div
\qquad \qquad
\calL := -\Div(\kappa \Grad ).
\end{equation} 
Note that we can regard the adjoint $\calD^*$ of $\calD$ as an operator acting on $\bbD$ upon identifying also this space with its dual $\bbD^*$ via the $L^2(\Domain)$-scalar product. Figure~\ref{F:abstract-spaces-diagram} summarizes the relation between the abstract spaces and operators. 
\begin{figure}[ht]
	\[
	\xymatrixcolsep{4pc}
	\xymatrixrowsep{4pc}
	\xymatrix{
		\bbU^* 
		\ar@{<-}[r]^{\calE}
		\ar@{<-}[d]^{\calD^*}
		&
		\bbU
		\ar[d]_{\calD}
		&
		L^2(\Domain)
		\ar[dl]_{\calP_\bbD}
		\ar[dr]^{\calP_{\overline{\bbP}}}
		&
		\bbP
		\ar[r]^{\calL}
		\ar[d]^{i}
		&
		\bbP^*
		\ar@{<-}[d]_{i^*}\\
		\bbD^*
		\ar@3{-}[r]
		&
		\bbD
		&&
		\overline{\bbP}
		\ar@3{-}[r]&
		\overline{\bbP}^*
	}
	\]
	\caption{\label{F:abstract-spaces-diagram} Spaces and operators describing the regularity in space for the weak formulation \eqref{E:BiotProblem-weak-formulation} of the Biot's equations. The symbol `$\equiv$' denotes the identification via the $L^2(\Domain)$-scalar product and $i$ is the embedding operator.}
\end{figure}
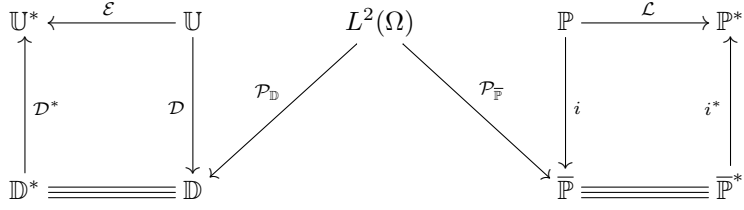

By comparing \eqref{E:abstract-operators-concrete} with the definition of the norms \eqref{E:norms}, we readily infer the identities
\begin{equation}
\label{E:coercivity-E-A}
\Norm{\cdot}_\bbU^2 = \left\langle \calE \cdot,\, \cdot \right\rangle_\bbU 
\qquad \text{and} \qquad 
\Norm{\cdot}_\bbP^2 = \left\langle \calL \cdot,\, \cdot\right\rangle_\bbP.
\end{equation}
Then, by duality, we obtain
\begin{equation}
\label{E:coercivity-E-A-duality}
\Norm{\cdot}^2_{\bbU^*} = \left\langle \cdot,\, \calE^{-1} \cdot \right\rangle_{\bbU} 
\qquad \text{and} \qquad
\Norm{\cdot}^2_{\bbP^*} = \left\langle \cdot,\, \calL^{-1} \cdot \right\rangle_\bbP.  
\end{equation}
Furthermore, the surjectivity of the divergence and the open mapping theorem \cite[Lemma~53.9]{Ern.Guermond:21b} imply
\begin{equation}
\label{E:inf-sup-B}
\dfrac{c}{\mu} \Norm{\cdot}_\Domain^2 \leq \Norm{\calD^* \cdot }_{\bbU^*}^2 \leq \dfrac{C}{\mu} \Norm{\cdot}_\Domain^2
\qquad \text{in $\bbD$}
\end{equation}
with constants $0 < c \leq C$ depending only on $\Domain$. 

After this preparation, we are in position to state the abstract weak formulation of the equations \eqref{E:BiotProblem-equations-fourfield}, with \eqref{E:BiotProblem-initialcondition} and \eqref{E:Biot-BCs-mixed}, as follows
\begin{equation}
\label{E:BiotProblem-weak-formulation}
\begin{alignedat}{2}
\calE u + \calD^* p_\Tot &= \ell_u \qquad && \text{in $L^2(\bbU^*)$}\\
\lambda \calD u - p_\Tot - \alpha \calP_\bbD p &= 0 \qquad && \text{in $L^2(\bbD)$}\\
\alpha \calP_{\overline \bbP} \calD u + \sigma p - m &= 0 \qquad && \text{in $L^2(\overline{\bbP})$}\\
\partial_t m + \calL p &= \ell_p \qquad  && \text{in $L^2(\bbP^*)$}\\
m(0) &= \ell_0 \qquad && \text{in $\bbP^*$}.
\end{alignedat}
\end{equation}
The loads $\ell_u$ and $\ell_p$ result from the data in the
equations~\eqref{E:BiotProblem-equations-fourfield} and in the
boundary conditions \eqref{E:Biot-BCs-mixed}. More precisely, they are
obtained as  
\begin{equation}
\label{E:abstract-data-concrete}
\begin{alignedat}{1}
\ell_u(v) = \int_{0}^T \left( \left\langle f_u,\, v \right\rangle_{\bbU} + \left\langle g_u,\, v \right\rangle_{H^{1/2}(\Gamma_{u,N})}  \right)\dt \\
\ell_p(n) = \int_{0}^T \left( \left\langle f_p,\, n \right\rangle_{\bbP} + \left\langle g_p,\, n \right\rangle_{H^{1/2}(\Gamma_{p,N})}  \right)\dt
\end{alignedat}
\end{equation}
for all $v \in L^2(\bbU)$ and $n \in L^2(\bbP)$.

We search for a solution of the equations \eqref{E:BiotProblem-weak-formulation} in the \emph{trial space} $\overline{\bbY}_1$, with
\begin{equation}
\label{E:trial-space}
\bbY_1 := L^2(\bbU) \times L^2(\bbD) \times L^2(\bbP) \times \left( L^2(\overline{\bbP}) \cap H^1(\bbP^*) \right). 
\end{equation}
The closure is taken with respect to the \emph{trial norm}
\begin{equation}
\label{E:trial-norm}
\begin{split}
&\| (\widetilde u, \widetilde p_\Tot, \widetilde p, \widetilde m) \|_1^2 := \\
&\qquad \int_{0}^{T} \left( \Norm{\widetilde u}^2_\bbU + \dfrac{1}{\mu}\Norm{\widetilde p_\Tot}_\Domain^2 + T\Norm{\partial_t \widetilde m + \calL \widetilde p}^2_{\bbP^*} \right)\dt + \Norm{\widetilde m(0)}^2_{\bbP^*}\\
&\quad  + \int_{0}^T \left( \dfrac{1 }{\mu + \lambda} \Norm{\lambda \calD
    \widetilde u - \widetilde p_\Tot - \alpha \calP_\bbD \widetilde p}^2_\Domain + \gamma\Norm{\alpha \calP_{\overline \bbP} \calD \widetilde u + \sigma
    \widetilde p - \widetilde m}^2_\Domain  \right)\dt
\end{split}
\end{equation}
where
\begin{align}\label{E:gamma}
  \gamma =
  \begin{cases}
    \min\left\lbrace  \dfrac{\mu + \lambda}{\alpha^2}, \dfrac{1}{\sigma}\right\rbrace   & \text{if $\sigma > 0$ and $\overline{\bbP} \subseteq \bbD$}
    \\[8pt]
    \dfrac{\mu + \lambda}{\alpha^2} + \dfrac{1}{\sigma} & \text{if $\sigma > 0$ and $\overline{\bbP} \nsubseteq \bbD$}\\[8pt]
    \dfrac{\mu + \lambda}{\alpha^2} &\text{if $\sigma = 0$}
  \end{cases}.
\end{align} 
We verify in
Proposition~\ref{P:rough-trial-functions} below that taking the
closure in the definition of the trial space is indeed 
necessary. Hereafter, we use the superscript `$\sim$' to distinguish a general
trial function in $\overline{\bbY}_1$ from the solution of the Biot's
equations. 

\begin{remark}[Definition of $\gamma$]
\label{R:gamma}
The reason for distinguishing three different cases in the definition
of $\gamma$ is made clear in the proof of the inf-sup stability in
Lemma~\ref{L:inf-sup} below. The first two cases are actually
equivalent under the condition $\sigma \eqsim \frac{\alpha^2}{\mu +
  \lambda}$, which is justified in
\cite[Section~2.2]{Lee.Mardal.Winther:17}, arguing by physical
principles. The condition characterizing the second case is equivalent
to prescribing $\sigma > 0$, $\Gamma_{u,E} = \partial \Domain$ and
$\Gamma_{p,N} \neq \partial \Domain$. This is, in a sense, the most
critical case, because the space $\overline{\bbP}$ changes from
$L^2(\Domain)$ to $L^2_0(\Domain)$ when passing from $\sigma > 0$ to
$\sigma = 0$. The unboundedness of $\gamma$ in the limit $\sigma \to
0$ reflects the lack of uniform control on the mean value of the
pressure. 
\end{remark}

The corresponding \emph{test space}, i.e., the space of all functions used to
test the weak formulation, is just the pre-dual of the product of the spaces
on the rightmost column in \eqref{E:BiotProblem-weak-formulation},
namely 
\begin{equation*}
\label{E:test-space}
\bbY_2 := L^2(\bbU) \times L^2(\bbD) \times L^2(\overline{\bbP}) \times L^2(\bbP) \times \bbP.
\end{equation*}
We equip $\bbY_2$ with the \emph{test norm}
\begin{equation}
\label{E:test-norm}
\begin{split}
\Norm{(v, q_\Tot, q, n, n_0)}_2^2 &:= \int_0^T \Big ( \Norm{v}^2_\bbU +  T^{-1}\Norm{n}^2_{\bbP} \Big)\dt + \Norm{n_0}^2_\bbP\\
& + \int_{0}^T \Big(  (\mu + \lambda) \Norm{q_\Tot}_\Domain^2 + \gamma^{-1}  \Norm{q}^2_\Domain  \Big)\dt .
\end{split} 
\end{equation}
The dual of $\Norm{\cdot}_2$ is the norm $\Norm{\cdot}_{2,*}$ mentioned in \eqref{E:linear-equations-stability} in the introduction, i.e. our measure of the regularity of the data. It is obtained as
\begin{equation}
\label{E:test-norm-dual}
\begin{split}
\Norm{(\widetilde \ell_{u}, \widetilde \ell_{p_\Tot}, \widetilde \ell_m, \widetilde \ell_p, \widetilde \ell_0)}_{2,*}^2 &= \int_0^T \Big ( \Norm{\widetilde \ell_u}^2_{\bbU^*} +  T\Norm{\widetilde \ell_p}^2_{\bbP^*} \Big)\dt + \Norm{\widetilde \ell_0}^2_{\bbP^*}\\
& + \int_{0}^T \left(  \dfrac{1}{\mu + \lambda} \Norm{\widetilde \ell_{p_\Tot}}_\Domain^2 + \gamma  \Norm{\widetilde \ell_m}^2_\Domain  \right)\dt 
\end{split} 
\end{equation}
for $(\widetilde \ell_{u}, \widetilde \ell_{p_\Tot}, \widetilde \ell_m, \widetilde \ell_p, \widetilde \ell_0) \in \bbY_2^*$. In \eqref{E:BiotProblem-weak-formulation}, the data components $\widetilde \ell_{p_\Tot}$ and $\widetilde \ell_m$ of the second and third equation vanish, but they may be nonzero, e.g. in a discretization method.

\begin{remark}[Two-field formulation]
\label{R:two-field-formulation}
As mentioned before, we could equivalently consider the original two-field formulation \eqref{E:BiotProblem-equations} of the Biot's equations in place of \eqref{E:BiotProblem-equations-fourfield}. In this case, the trial space to be used for the abstract weak formulation is the closure of
\begin{equation*}
\label{E:trial-space-2field}
\Big\{(\widetilde u,\widetilde p) \in L^2(\bbU) \times L^2(\bbP) \mid (\alpha \calP_{\overline \bbP}\calD \widetilde u + \sigma \widetilde p) \in H^1(\bbP^*) \Big\} 
\end{equation*}
with respect to the trial norm
\begin{equation}
\label{E:trial-norm-2field}
\begin{split}
&\int_{0}^{T} \left( \Norm{\widetilde u}^2_\bbU + \dfrac{1}{\mu}\Norm{\lambda \calD \widetilde u - \alpha \calP_\bbD\widetilde p}_\Domain^2 + T\Norm{\partial_t (\alpha \calP_{\overline \bbP}\calD \widetilde u + \sigma  \widetilde p) + \calL \widetilde p}^2_{\bbP^*} \right)\dt \\ &\qquad  +  \Norm{(\alpha \calP_{\overline \bbP}\calD \widetilde u + \sigma  \widetilde p)_{|t=0}}^2_{\bbP^*}.
\end{split}
\end{equation}
The corresponding test space is
\begin{equation}
\label{E:test-space-2field}
L^2(\bbU) \times L^2(\bbP) \times \bbP
\end{equation}
equipped with the test norm
\begin{equation}
\label{E:test-norm-2field}
\int_{0}^T \left( \Norm{v}^2_\bbU  + T^{-1}\Norm{n}^2_{\bbP} \right)\dt + \Norm{n_0}^2_\bbP.
\end{equation}
We opted for the four-field formulation because it appears to simplify some arguments in the proof of Lemma~\ref{L:nondegeneracy} (nondegeneracy) in Section~\ref{SS:Nondegeneracy}.
\end{remark}

\begin{remark}[Norms]
\label{R:norms}
To justify the norms \eqref{E:trial-norm} and \eqref{E:test-norm} for the analysis of the four-field formulation \eqref{E:BiotProblem-weak-formulation}, it is convenient starting from the
two-field formulation discussed in
Remark~\ref{R:two-field-formulation}. The test norm
\eqref{E:test-norm-2field} is simply the product norm on the test
space \eqref{E:test-space-2field}. Prescribing it corresponds to prescribing the regularity of the
data. Then, the expression of the trial norm
\eqref{E:trial-norm-2field} is not at our disposal, because it is
determined (up to norm equivalence) by requiring that the trial space is isomorphic to the dual of the test space through the weak formulation \eqref{E:BiotProblem-weak-formulation}, cf. \eqref{E:linear-equations-stability}. When passing from the two- to the four-field formulation, two additional terms must be included in the definition of both \eqref{E:trial-norm} and \eqref{E:test-norm},
because inhomogeneous data $\ell_{p_\Tot}$ and $\ell_m$ are in principle allowed in the second and
third equations of \eqref{E:BiotProblem-weak-formulation}.  
Finally, let us mention that the various terms in \eqref{E:trial-norm} and \eqref{E:test-norm} are scaled so that all hidden constants hereafter are independent of the final time $T$ and of the material parameters. In particular, the scaling guarantees that the definition of each norm is consistent in terms of physical units for $\Dim = 3$.
\end{remark}

\begin{remark}[Regularity of $\ell_u$]
\label{R:regularity-lu}
Owing to the definition of the norm $\Norm{\cdot}_{2,*}$ in \eqref{E:test-norm-dual} we assume that the load in the first equation of the weak formulation \eqref{E:BiotProblem-weak-formulation} is such that $\ell_u \in L^2(\bbU^*)$, cf. Theorem~\ref{T:well-posedness} below. To our best knowledge, this relaxes the regularity assumption $\ell_u \in H^1(\bbU^*)$ made in all previous references, see e.g. \cite{Showalter:00,Zenisek:84}.
\end{remark}

\section{Well-posedness of the weak formulation}
\label{S:WellPosedness}

In this section we prove that the equations
\eqref{E:BiotProblem-weak-formulation} are well-posed by means of the
inf-sup theory. To this end, it is convenient rewriting the equations
in the following variational form: find $y_1 \in \overline{\bbY}_1$
such that

\begin{equation}
\label{E:BiotProblem-variational}
b(y_1, y_2) = \ell(y_2) \qquad \forall y_2 \in \bbY_2. 
\end{equation}
The bilinear form $b: \overline{\bbY}_1 \times \bbY_2 \to \R$ and the load $\ell: \bbY_2 \to \R$ are defined as 
\begin{equation}
\label{E:BiotProblem-abstract-form}
\begin{split}
b(\widetilde y_1, y_2&) = \int_0^T \Big( \left\langle \calE \widetilde u +
    \calD^* \widetilde p_\Tot,\, v \right\rangle_\bbU + \left\langle
    \partial_t \widetilde m + \calL \widetilde p,\, n \right\rangle_\bbP
\Big)\dt + \left\langle \widetilde m(0),\, n_0 \right\rangle_\bbP
\\
&+ \int_0^T \Big ( \left( \lambda \calD \widetilde u - \widetilde
    p_\Tot - \alpha \calP_\bbD\widetilde p,\, q_\Tot\right)_\Domain +
  \left( \alpha \calP_{\overline \bbP}\calD \widetilde u + \sigma \widetilde p - \widetilde m,\, q
  \right)_\Domain   \Big )\dt 
\end{split} 
\end{equation}
and
\begin{equation}
\label{E:BiotProblem-abstract-load}
\ell(y_2) = \ell_u(v) + \ell_p(n) + \left\langle \ell_0,\, n_0 \right\rangle_\bbP 
\end{equation}
for $\widetilde y_1 = (\widetilde u, \widetilde p_\Tot, \widetilde p,
\widetilde m) \in \overline{\bbY}_1$ and $y_2 = (v, q_\Tot, q, n, n_0)
\in \bbY_2 $. The projections $\calP_\bbD$ and $\calP_{\overline
  \bbP}$ could be neglected in the definition of $b$, because they are
tested with functions from their respective range. 

According to the so-called Banach-Ne\u{c}as-Babu\v{s}ka theorem, the
well-posedness of the problem \eqref{E:BiotProblem-variational}
is equivalent to the three properties
verified in the next lemmas.  

\begin{lemma}[Boundedness]
\label{L:boundedness} The bilinear form $b$ in
\eqref{E:BiotProblem-abstract-form} is such that 
\begin{equation}
\label{E:boundedness}
\sup_{y_2 \in \bbY_2} \dfrac{b(\widetilde y_1, y_2)}{\Norm{y_2}_2} \lesssim \| \widetilde y_1 \|_1
\end{equation} 
for all $\widetilde y_1 \in \overline{\bbY}_1$. The hidden constant
depends only on the domain $\Domain$. 
\end{lemma}

\begin{proof}
The claimed bound readily follows from the Cauchy-Schwartz inequality,
the identities in \eqref{E:coercivity-E-A} and the second part of
\eqref{E:inf-sup-B}. 
\end{proof}

\begin{lemma}[Inf-sup stability]
\label{L:inf-sup}
The bilinear form $b$ in \eqref{E:BiotProblem-abstract-form} is such that
\begin{equation}
\label{E:inf-sup}
\sup_{y_2 \in \bbY_2} \dfrac{b(\widetilde y_1, y_2)}{\Norm{y_2}_2} \gtrsim \left( \| \widetilde y_1 \|_1^2 +
  \Norm{\widetilde m}^2_{L^\infty(\bbP^*)} + \int_{0}^T \left( \lambda
    \Norm{\calD \widetilde u}^2_\Domain + \frac{1}{\gamma} \Norm{\widetilde p
    }^2_\Domain\right) \dt \right)^{\frac{1}{2}}  
\end{equation} 
for all $\widetilde y_1  = (\widetilde u, \widetilde p_\Tot, \widetilde p, \widetilde m)
\in \overline{\bbY}_1$. The hidden constant depends only on the domain~$\Domain$. 
\end{lemma}

\begin{proof}
See Section~\ref{SS:InfSup}.
\end{proof}

The combination of the Lemmas~\ref{L:boundedness} and \ref{L:inf-sup}
points out that the trial norm $\Norm{\cdot}_1$ is actually equivalent
to the (stronger) norm on the right-hand side of \eqref{E:inf-sup}. We
deduce a first relevant consequence of this observation in the following remark. 

\begin{remark}[Continuity of the total fluid content]
\label{R:continuity-m}
The embedding $ H^1(\bbP^*) \subseteq C^0(\bbP^*)$ reveals that the linear operator
\begin{equation*}
\label{E:continuity-m}
\bbY_1 \ni \widetilde{y}_1 = (\widetilde u, \widetilde p_\Tot, \widetilde p, \widetilde m) \mapsto \widetilde{m} \in C^0(\bbP^*)
\end{equation*}
is well-defined. The combination of Lemmas~\ref{L:boundedness} and \ref{L:inf-sup} further implies that this operator is bounded, i.e.
\begin{equation}
\label{E:continuity-m-bound}
\Norm{\widetilde m}_{L^\infty(\bbP^*)} \lesssim \Norm{\widetilde y_1}_1.
\end{equation}
Therefore we can extend it from $\bbY_1$ to $\overline{\bbY}_1$ by density. This observation is important to guarantee that the initial condition in \eqref{E:BiotProblem-weak-formulation} is
meaningful.   
\end{remark}

\begin{lemma}[Nondegeneracy]
\label{L:nondegeneracy}
Let the bilinear form $b$ be as in \eqref{E:BiotProblem-abstract-form} and assume
\begin{equation}
\label{E:nondegeneracy}
b(\widetilde y_1, y_2) = 0
\end{equation}
for all $\widetilde y_1 \in \overline{\bbY}_1$ and for some $y_2 \in \bbY_2$. Then, we have $y_2 = 0$.
\end{lemma}

\begin{proof}
See Section~\ref{SS:Nondegeneracy}.
\end{proof}

The combination of the above lemmas implies our main result, stating
existence, uniqueness and two-sided stability of the solution of the equations
\eqref{E:BiotProblem-weak-formulation}, in the sense of \eqref{E:linear-equations-stability}. In particular, the latter property
ensures that the regularity established for the solution matches the regularity requirements for the data.  

\begin{theorem}[Well-posedness]
\label{T:well-posedness}
For all $(\ell_u, \ell_p, \ell_0) \in L^2(\bbU^*) \times L^2(\bbP^*) \times \bbP^*$, the equations \eqref{E:BiotProblem-weak-formulation} have a unique solution $y_1 = (u,p_\Tot, p, m) \in \overline{\bbY}_1$, which fulfills the two-sided stability bounds
\begin{equation}
\label{E:well-posedness}
\begin{split}
&\Norm{y_1}_1^2 + \Norm{m}^2_{L^\infty(\bbP^*)} + \int_{0}^T \left( \lambda \Norm{\calD u}^2_\Domain + \frac{1}{\gamma} \Norm{p}^2_\Domain \right)\dt\\ 
& \qquad \qquad \eqsim \int_{0}^T \left( \Norm{u}^2_\bbU + \dfrac{1}{\mu} \Norm{p_\Tot}^2_\Domain +  T\Norm{\partial_t m + \calL p}^2_{\bbP^*} \right)\dt + \Norm{\widetilde m(0)}^2_{\bbP^*}\\
& \qquad \qquad \eqsim \int_{0}^T \left( \Norm{\ell_u}^2_{\bbU^*} + T\Norm{\ell_p}^2_{\bbP^*} \right)\dt + \Norm{\ell_0}^2_{\bbP^*}.
\end{split}
\end{equation}
The hidden constants depend only on the domain $\Domain$ and we have $m \in C^0(\bbP^*)$.
\end{theorem}

\begin{proof}
The existence and the uniqueness of the solution result from the
combination of the Banach-Ne\u{c}as-Babu\v{s}ka theorem
\cite[Theorem~25.9]{Ern.Guermond:21b} with the
Lemmas~\ref{L:boundedness}, \ref{L:inf-sup} and
\ref{L:nondegeneracy}. The same argument
\cite[Remark~25.12]{Ern.Guermond:21b} yields 
\begin{equation*}
\label{E:well-posedness-proof}
\Norm{(u,p_\Tot, p, m)}_1 \eqsim \Norm{(\ell_u, 0,0,\ell_p, \ell_0)}_{2,*}
\end{equation*}
where the hidden constants depend only on the constants in
\eqref{E:boundedness} and \eqref{E:inf-sup}. Then, the claimed
two-sided stability bound follows by the definitions
\eqref{E:trial-norm} and \eqref{E:test-norm-dual} of the norms
$\Norm{\cdot}_1$ and $\Norm{\cdot}_{2,*}$ and the equivalence of
$\Norm{\cdot}_1$ with the stronger norm on the right-hand side of
\eqref{E:inf-sup} and by recalling that $(u,p_\Tot, p, m)$ solves
\eqref{E:BiotProblem-weak-formulation}. Finally, the continuity in
time of the component $m$ of the solution is guaranteed by
Remark~\ref{R:continuity-m}.  
\end{proof}

Before examining the proof of Lemmas~\ref{L:inf-sup}-\ref{L:nondegeneracy}, it is worth comparing Theorem~\ref{T:well-posedness} with related results in the literature.

\begin{remark}[Comparison with \cite{Khan.Zanotti:22}]
	\label{R:comparison-Khan-Zanotti}
The stability bound in Theorem~\ref{T:well-posedness} is consistent
with the one established in \cite[Section~2]{Khan.Zanotti:22} for the
stationary equations obtained from \eqref{E:BiotProblem-equations} by
semi-discretization in time with the backward Euler scheme. Indeed, in
that context, the stability estimate involves the trial norm 
\begin{equation*}
\label{E:trial-norm-stationary}
\tau\left( \Norm{u}^2_\bbU + \lambda \Norm{\calD u}^2_\Domain + \dfrac{1}{\mu} \Norm{p_\Tot}^2_\Domain + \sigma \Norm{p}^2_\Domain + \tau \Norm{p}^2_\bbP \right) + \Norm{m}^2_{\bbP^*}
\end{equation*}
where $\tau$ denotes the time step. By interpreting the multiplication
by $\tau$ as a kind of time integration, we see that each term here
has a corresponding one in
\eqref{E:well-posedness}. The only exception is the $\bbP$-norm of
$p$, which is multiplied by an additional factor $\tau$. Hence, we
cannot expect a uniform control on the $L^2(\bbP)$-norm of $p$ in
terms of the left-hand side of \eqref{E:well-posedness} in the limit
$\tau \to 0$, i.e., when passing from the time-semidiscretization to
the original equations. In Proposition~\ref{P:rough-trial-functions}
below we verify our expectation by means of a counterexample. 
\end{remark}

\begin{remark}[Comparison with \cite{Zenisek:84}]
\label{R:comparison-Zenisek}
The technique introduced by \v{Z}en\'{\i}\v{s}ek
\cite{Zenisek:84} consists in applying the so-called Feado-Galerkin
scheme and in establishing a stability estimate that serves to infer
the existence and the uniqueness of the solution by testing the
equations in \eqref{E:BiotProblem-equations} with
$(\partial_t u, p)$. This ultimately provides a
one-sided stability bound involving the following norm of the solution 
\begin{equation*}
\label{E:trial-norm-Zenisek}
\Norm{u}^2_{L^\infty(\bbU)} + \lambda \Norm{\calD
  u}^2_{L^\infty(L^2(\Domain))} + \sigma
\Norm{p}^2_{L^\infty(L^2(\Domain))} + \int_{0}^{T} \Norm{p}^2_\bbP
\dt; 
\end{equation*}  
see \cite[Theorem~1]{Zenisek:84} (the estimate is only established in
the proof). The scaling with respect
to the parameters is the same as in \eqref{E:well-posedness}. Each
term is measured in a stronger norm because of the higher regularity
assumption $\ell_u \in H^1(\bbU^*)$, cf. Remark~\ref{R:regularity-lu}. Still, in contrast to \eqref{E:well-posedness}, the bound in \cite{Zenisek:84} (and similar ones) cannot be reversed, meaning that the regularity established for the solution does not  match the regularity requirement for the data. 
\end{remark}

\begin{remark}[Comparison with \cite{Li.Zikatanov:22}]
\label{R:comparison-Li-Zikatanov}
Li and Zikatanov \cite{Li.Zikatanov:22} assume even more regular data
than in \cite{Zenisek:84}, namely $\ell_u \in H^1(\bbU^*)$ and $\ell_p
\in L^2(L^2(\Domain))$, as well as smooth and compatible initial data,
cf. Remark~\ref{R:initial-condition}. Then, by improving on the
original technique of \v{Z}en\'{\i}\v{s}ek, they are able to establish
the regularity 
\begin{equation*}
\label{E:comparison-Li-Zikatanov}
u \in H^1(\bbU) \qquad\text{and}\qquad p \in H^1(\bbP) \cap L^2(\mathbb
L)
\end{equation*}
where $\mathbb{L}=\{ \widetilde p \in \bbP \mid \calL \widetilde p
\in\overline{\bbP} \}$. They establish also a corresponding stability
estimate, where the scaling with respect to the material parameters is
similar as in the above-mentioned results. Remarkably, such estimate
can be reversed, in the sense that it fulfills an equivalence like
\eqref{E:linear-equations-stability}. 
\end{remark}

We end the discussion on Theorem~\ref{T:well-posedness} with a remark illustrating the robustness in the incompressible regime.

\begin{remark}[Incompressible elastic materials]
\label{R:incompressibile-elasticity}
Owing to the robustness of the hidden constants in Theorem~\ref{T:well-posedness}, our framework applies also to exactly incompressible materials, i.e. for $\lambda = +\infty$. In this case, the equivalence \eqref{E:well-posedness} implies $\calD u = 0$ in $L^2(\bbD)$. This condition replaces the second equation in \eqref{E:BiotProblem-weak-formulation}, which is no longer needed, because the total pressure is entirely characterized by the first equation. Thus, the first two equations in \eqref{E:BiotProblem-weak-formulation} are reduced to the Stokes equations pointwise in time and are decoupled from the remaining ones, which correspond to a nonstandard mixed formulation of the heat equation with initial value in $\bbP^*$. Note, in particular, that the third equation becomes $m = \sigma p$ and it remains the only source of control over the $L^2(L^2(\Domain))$-norm of the pressure. This observation and the definition of $\gamma$ in \eqref{E:gamma} explain why Theorem~\ref{T:well-posedness} does not control the $L^2(L^2(\Domain))$-norm of the pressure in the limit $\gamma \to +\infty$.
\end{remark}

\subsection{Inf-sup stability}
\label{SS:InfSup}

This section is devoted to the proof of Lemma~\ref{L:inf-sup}. For
this purpose, let $\widetilde y_1 = (\widetilde u, \widetilde p_\Tot,
\widetilde p, \widetilde m) \in \bbY_1$. We shall construct a test
function $y_2 \in \bbY_2$ such that  
\begin{equation}
\label{E:inf-sup-proof-lower-bound}
b(\widetilde y_1, y_2) \gtrsim 
\| \widetilde y_1 \|_1^2 + \Norm{\widetilde m}^2_{L^\infty(\bbP^*)} +
\int_{0}^T \left( \lambda \Norm{\calD \widetilde u}^2_\Domain + \frac{1}{\gamma}
  \Norm{\widetilde p }^2_\Domain\right) \dt  
\end{equation}
as well as
\begin{equation}
\label{E:inf-sup-proof-upper-bound}
\Norm{y_2}^2_2 \lesssim \left( \| \widetilde y_1 \|_1^2 + \Norm{\widetilde m}^2_{L^\infty(\bbP^*)} +\lambda\int_0^T\Norm{\calD \widetilde u}^2_\Domain  \right).
\end{equation}
The combination of these inequalities with a density argument implies
\eqref{E:inf-sup}. Taking $\widetilde y_1$ in $\bbY_1$ (and not
directly in $\overline{\bbY}_1$) simplifies the derivation of
\eqref{E:inf-sup-proof-lower-bound} and, in particular, the lower
bound of the term $\mathfrak{I}_2$ below. 

Let $s \in [0, T]$ be a value to be specified later. We denote by
$\chi_s:[0,T] \to \R$ the indicator function on $[0, s]$. In other
words, $\chi_s(t)$ equals $1$ for $t\leq s$ and it vanishes for $t >
s$. We consider the test function $y_{2,s} = (v, q_\Tot, q, n, n_0) \in \bbY_2$ defined by 
\begin{equation}
\label{E:inf-sup-test-function-components}
\begin{alignedat}{2}
v &= \Big(\widetilde u + \calE^{-1} \calD^*\widetilde p_\Tot)\chi_s \quad && \in L^2(\bbU)
\\
q_\Tot &= \Big(\dfrac{3\max\{1,C\}}{\mu + \lambda}(\lambda \calD \widetilde u - \widetilde p_\Tot - \alpha
\calP_\bbD\widetilde p) - \frac{\alpha K}{\mu + \lambda} \calP_\bbD \widetilde p \Big)\chi_s \quad &&\in
L^2(\bbD)
\\
q &= \dfrac{3 \gamma}{K} (\alpha \calP_{\overline \bbP}\calD
 \widetilde u + \sigma \widetilde p - \widetilde m) \chi_s \quad && \in L^2(\overline \bbP)
\\
n &= \calL^{-1} (2 \widetilde m + s(\partial_t \widetilde m + \calL \widetilde p))\chi_s \quad && \in L^2(\bbP^*)
\\
n_0 &= 2\calL^{-1} \widetilde m(0) && \in \bbP^*
\end{alignedat}
\end{equation} 
with the constants $C$ from \eqref{E:inf-sup-B}, $\gamma$ from \eqref{E:gamma} and $K>0$ to be determined later.

\begin{remark}[Motivating the test function]
\label{R:test-function}
The definition \eqref{E:BiotProblem-abstract-form} of the form $b$
consists of five summands. Roughly speaking, the test function
$y_{2,s}$ is designed so as to obtain a positive contribution (i.e., a
squared norm) from each summand. The second and the third components
of $y_{2,s}$ are additionally scaled by `sufficiently large'
constants, so as to compensate negative contributions arising from the
analysis of the other terms. Furthermore, the third and fourth components of $y_{2,s}$ contain additional terms to control
the  variables $\widetilde p$ and $\widetilde m$. The choice of the fourth and the
fifth components is in line with the derivation of the inf-sup stability for scalar parabolic equations,
see e.g. the proof of \cite[Lemma~71.2]{Ern.Guermond:21c}.  
\end{remark}

The claimed regularity of the components of the test function entails that we indeed have $y_{2,s} \in
\bbY_2$. Such regularity can be verified by recalling the setting in
Section~\ref{SS:weak-formulation}. In particular, we mention that the
operator $\calE^{-1} \calD^*$ maps \(\bbD\) into $\bbU$ and that $\widetilde m(0) \in \bbP^*$ is interpreted by means of the embedding $H^1(\bbP^*) \subseteq C^0(\bbP^*)$. Finally, the multiplication of all components (except the last one) by $\chi_s$ is admissible, because the test space
$\bbY_2$ involves only $L^2$ regularity in time. 

In order to establish \eqref{E:inf-sup-proof-lower-bound} for a
suitable test function $y_2$, we investigate the action of the form
$b$ onto the pair $(\widetilde y_1, y_{2,s})$. By recalling the
definition \eqref{E:BiotProblem-abstract-form} of $b$, we infer 
\begin{equation}
\label{E:inf-sup-test-function-action}
\begin{alignedat}{4}
b(\widetilde y_1, y_{2,s}) &= \int_0^s \left\langle \calE \widetilde u
  + \calD^* \widetilde p_\Tot, \,\widetilde u + \calE^{-1} \calD^*
  \widetilde p_\Tot \right\rangle_\bbU \dt &  \qquad (=: \mathfrak
I_1)
\\  
& + 2\int_{0}^s \left\langle \partial_t \widetilde m + \calL
  \widetilde p,\, \calL^{-1} \widetilde m \right\rangle_\bbP \dt &
\qquad (=: \mathfrak I_2)
\\ 
&+ s\int_{0}^s \left\langle \partial_t \widetilde m + \calL \widetilde
  p,\, \calL^{-1} (\partial_t \widetilde m + \calL \widetilde p)
\right\rangle_\bbP \dt &\qquad (=: \mathfrak I_3)
\\ 
&+ 2\left\langle \widetilde m(0),\, \calL^{-1} \widetilde
  m(0)\right\rangle_{\bbP} & \qquad (=: \mathfrak I_4)
\\
&+ \dfrac{3\max\{1,C\}}{\mu +\lambda} \int_0^s\Norm{ \lambda \calD
  \widetilde u  - \widetilde p_\Tot - \alpha \calP_\bbD\widetilde
  p}_\Domain^2 \dt
\\
&- \dfrac{K\alpha}{\mu + \lambda} \int_0^s ( \lambda \calD \widetilde u - \widetilde p_\Tot - \alpha \calP_\bbD \widetilde p, \calP_\bbD \widetilde p  )_\Domain \dt & \qquad (=: \mathfrak I_5)\\
& + \dfrac{3 \gamma}{K} \int_0^s \Norm{\alpha \calP_{\overline \bbP}\calD \widetilde u + \sigma \widetilde p - \widetilde m}^2_\Domain \dt.
\end{alignedat}
\end{equation}
We further investigate the five terms $\mathfrak{I}_1, \dots,
\mathfrak{I}_5$ on the right-hand side. The first parts of \eqref{E:coercivity-E-A} and \eqref{E:coercivity-E-A-duality} imply
\begin{equation*}
\mathfrak{I}_1 = \int_{0}^s \Big( \Norm{\widetilde u}_{\bbU}^2 + 2
\left\langle \calD^*\widetilde p_\Tot, \,\widetilde u \right\rangle_\bbU
+ \Norm{\calD^* \widetilde p_\Tot}_{\bbU^*}^2 \Big)\dt. 
\end{equation*}
Regarding the second summand on the right-hand side, we have
\begin{multline*}
% \begin{split}
\left\langle \calD^* \widetilde p_\Tot,\, \widetilde u  \right\rangle_\bbU 
= 
%\left\langle \calD \widetilde u, \widetilde p_\Tot \right\rangle_\Domain 
%= 
( \calD \widetilde u, \,\widetilde p_\Tot - \lambda \calD \widetilde u +
  \alpha \calP_\bbD\widetilde p )_\Domain + \lambda \Norm{\calD
  \widetilde u}^2_\Domain - \alpha (  \calD \widetilde u, \calP_\bbD
  \widetilde p)_\Domain
\\
\geq \dfrac{3\lambda}{4} \Norm{\calD \widetilde u}^2_\Domain - \dfrac{1}{4} \Norm{\widetilde{u}}^2_\bbU -
\dfrac{\max\{1,C\}}{\mu + \lambda} \Norm{\lambda \calD \widetilde u - \widetilde p_\Tot -
  \alpha \calP_\bbD\widetilde p}^2_\Domain 
-\alpha ( \calD \widetilde u, \,\widetilde p )_\Domain  
%\end{split}
\end{multline*}
according to Young's inequality and the bound $\mu\Norm{\calD \widetilde u}^2_\Domain\le C \Norm{\widetilde{u}}^2_\bbU$, which follows from the second bound in~\eqref{E:inf-sup-B} and the fact that the operator norms of $\calD$ and $\calD^*$ coincide. Notice that we could neglect the
projection $\calP_\bbD$ in the last summand, because of the inclusion
$\calD \widetilde u \in \bbD$. We insert this lower bound into the
previous identity. By invoking also the lower bound in \eqref{E:inf-sup-B}, we obtain 
\begin{equation*}
\begin{split}
\mathfrak{I}_1 \geq \int_0^s \Bigg ( &\dfrac{1}{2}\Norm{\widetilde u}^2_\bbU + \dfrac{3\lambda}{2}
\Norm{\calD \widetilde u}^2_\Domain + \dfrac{c}{\mu}\Norm{\widetilde
  p_\Tot}^2_\Domain\\   &- \dfrac{2\max \{1,C\}}{\mu + \lambda}\Norm{\lambda \calD \widetilde u -
  \widetilde p_\Tot - \alpha \calP_\bbD\widetilde p}^2_\Domain 
-2\alpha ( \calD \widetilde u,\, \widetilde p )_\Domain  \Bigg)\dt.
\end{split}
\end{equation*}
The next term to be considered is 
\begin{equation*}
\mathfrak{I}_2 = 2 \int_{0}^s \Big( \left\langle \partial_t \widetilde m,\,
  \calL^{-1} \widetilde m \right\rangle_\bbP + ( \widetilde m,\,
  \widetilde p )_\Domain  \Big)\dt.  
\end{equation*}
The second part of \eqref{E:coercivity-E-A-duality} and an integration by parts \cite[Lemma~64.40]{Ern.Guermond:21c} reveal
\begin{equation*}
\int_{0}^s \left\langle \partial_t \widetilde m, \,\calL^{-1} \widetilde m
\right\rangle_{\bbP^*} \dt = \dfrac{1}{2} \Norm{\widetilde
  m(s)}^2_{\bbP^*} - \dfrac{1}{2} \Norm{\widetilde m(0)}^2_{\bbP^*}.  
\end{equation*} 
We bound the other summand in $\mathfrak{I}_2$ from below using Young's inequality to obtain
\begin{equation*}
\begin{split}
( \widetilde m,\, \widetilde p)_\Domain 
&= 
( \widetilde m - \alpha \calP_{\overline \bbP} \calD \widetilde u - \sigma \widetilde p,\, \widetilde p )_\Domain + \alpha ( \calD
\widetilde u,\, \widetilde p )_\Domain + \sigma \Norm{\widetilde
	p}^2_\Domain 
\\
&\geq  \alpha ( \calD \widetilde u,\, \widetilde p )_\Domain + \left(\sigma - \frac{K}{4\gamma}\right) \Norm{\widetilde p}^2_\Domain - \dfrac{\gamma}{K} \Norm{\alpha \calP_{\overline \bbP}\calD \widetilde u + \sigma \widetilde p - \widetilde m}^2_\Domain.
\end{split}
\end{equation*}
Combining this estimate with the above identities, we infer 
\begin{equation*}
\begin{split}
\mathfrak{I}_2 
&\geq 
\Norm{\widetilde m(s)}^2_{\bbP^*} - \Norm{\widetilde m(0)}^2_{\bbP^*}\\ 
&+ \int_0^s \Big( 2\alpha ( \calD \widetilde u,\, \widetilde p )_\Domain + \Big(2\sigma - \frac{K}{2\gamma}\Big)\Norm{\widetilde p}^2_\Domain
- \dfrac{2\gamma}{K} \Norm{\alpha \calP_{\overline \bbP}\calD \widetilde u + \sigma \widetilde p - \widetilde m}^2_\Domain \Big)\dt.
\end{split} 
\end{equation*}

\begin{remark}[Critical terms]
\label{R:critical-terms}
It is worth noticing that the (non necessarily positive) term $2\alpha (\calD \widetilde u, \widetilde p )_\Domain$ in the lower bound of $\mathfrak{I}_2$ is compensated by the corresponding term $-2\alpha ( \calD \widetilde u, \widetilde p )_\Domain$ in the lower bound of $\mathfrak{I}_1$. A similar compensation, obtained by a different test function, underlines the proof of the stability estimate established by \v{Z}en{\'\i}{\v{s}}ek \cite{Zenisek:84} and later used by several other authors. 
\end{remark}

According to the second
identity in \eqref{E:coercivity-E-A-duality}, we rewrite the third and
the fourth terms as 
\begin{equation*}
\mathfrak{I}_3 = s\int_0^s \Norm{\partial_t \widetilde m + \calL \widetilde p}^2_{\bbP^*}\dt
\qquad \text{and} \qquad
\mathfrak{I}_4 = 2\Norm{\widetilde m(0)}^2_{\bbP^*}.
\end{equation*}
For the fifth term, we exploit Young's inequality once more to obtain
\begin{equation*}
   \mathfrak I_5 
   \geq 
   \int_0^s \left( -\lambda \Norm{\calD \widetilde u}^2_\Domain - \frac{c}{2\mu} \Norm{\widetilde p_\Tot}^2_\Domain + \frac{\alpha^2 K}{\mu + \lambda} \Big( 1 - \frac{K}{4} - \frac{K}{2c} \Big) \Norm{\calP_\bbD \widetilde p}^2_\Domain \right )\dt.
\end{equation*}
Finally, we notice that, upon choosing $K>0$ small enough (only depending on the constant $c$ in the first part of \eqref{E:inf-sup-B}),
\begin{equation*}
    \Big( 2\sigma - \frac{K}{2\gamma} \Big) \Norm{\widetilde p}^2_\Domain + \frac{\alpha^2 K}{\mu + \lambda} \Big( 1 - \frac{K}{4} - \frac{K}{2c} \Big) \Norm{\calP_\bbD \widetilde p}^2_\Domain \geq \frac{K}{4\gamma} \Norm{\widetilde p}^2_\Domain. 
\end{equation*}
This bound follows from the definition of $\gamma$ in \eqref{E:gamma}, upon noticing that we have $\calP_\bbD \widetilde p = \widetilde p$ whenever $\overline{\bbP} \subseteq \bbD$. This holds true, in particular, for $\sigma = 0$.

We insert the identities for $\mathfrak{I}_3$ and $\mathfrak{I}_4$,
the lower estimates for $\mathfrak{I}_1$, $\mathfrak{I}_2$ and $\mathfrak I_5$ and the latter bound above into
\eqref{E:inf-sup-test-function-action}. We obtain 
\begin{equation*}
\label{E:inf-sup-test-function-action-2}
\begin{split}
b(\widetilde y_1, y_{2,s}) &\geq \int_0^s \Bigg( \dfrac{1}{2}\Norm{\widetilde u}^2_\bbU +
\dfrac{\lambda}{2} \Norm{\calD \widetilde u}^2_\Domain +
\dfrac{c}{2\mu} \Norm{\widetilde p_\Tot}^2_\Domain + \frac{K}{4\gamma}
\Norm{\widetilde p}^2_\Domain +  s\Norm{\partial_t \widetilde m + \calL \widetilde
  p}^2_{\bbP^*}  \Bigg)\dt
\\ 
&\quad+ \Norm{\widetilde m(s)}^2_{\bbP^*} + \Norm{\widetilde m(0)}^2_{\bbP^*}
\\
&\quad+ \int_0^s \Big( \dfrac{1}{\mu + \lambda} \Norm{\lambda \calD \widetilde u -
  \widetilde p_\Tot - \alpha \calP_\bbD \widetilde p}^2_\Domain + \frac{\gamma}{K} \Norm{\alpha \calP_{\overline \bbP}\calD
  \widetilde u + \sigma \widetilde p - \widetilde m}^2_\Domain \Big)\dt. 
\end{split}
\end{equation*}
We conclude that the lower bound \eqref{E:inf-sup-proof-lower-bound}
holds true and the hidden constant depends only on $\Domain$ through $c$ and $C$ in \eqref{E:inf-sup-B}, provided
that we select the test function 
\begin{equation} 
\label{E:inf-sup-test-function}
y_2 = y_{2,T} + y_{2, \overline s} 
\end{equation}
where $\overline s \in [0, T]$ is chosen so that $\Norm{\widetilde
  m(\overline s)}_{\bbP^*} = \Norm{\widetilde
  m}_{L^\infty(\bbP^*)}$.

The last step of the proof consists in bounding the norm of the test
function $y_2$ in \eqref{E:inf-sup-test-function}, in order to verify
\eqref{E:inf-sup-proof-upper-bound}. To this end, we establish a
corresponding upper bound for the norm of any function $y_{2,s}$, that
is uniform with respect to the parameter $s \in [0,T]$. According to
the definition \eqref{E:test-norm} of the test norm, the identities \eqref{E:coercivity-E-A} and
\eqref{E:coercivity-E-A-duality} and the second part of
\eqref{E:inf-sup-B}, the components of $y_{2,s}$ in \eqref{E:inf-sup-test-function-components} are such that
\begin{equation*}
\begin{alignedat}{1}
\int_0^T\Norm{v}_\bbU^2 \dt &\lesssim \int_0^T \Big(\Norm{\widetilde u}^2_\bbU + \frac{1}{\mu}\Norm{\widetilde p_\Tot}_\Domain^2 \Big)\dt
\\
(\mu +\lambda)\int_0^T \Norm{q_\Tot}_\Domain^2 &\lesssim  \int_0^T \Big( \dfrac{1}{\mu + \lambda}\Norm{\lambda \calD \widetilde u - \widetilde p_\Tot - \alpha
\calP_\bbD\widetilde p}_\Domain^2 +\lambda \Norm{\calD \widetilde u}^2_\Domain + \frac{1}{\mu}\Norm{\widetilde p_\Tot}^2_\Domain\Big)\dt
\\
\frac{1}{\gamma}\int_0^T \Norm{q}_\Domain^2\dt &\lesssim \gamma \int_0^T \Norm{\alpha \calP_{\overline \bbP}\calD
 \widetilde u + \sigma \widetilde p - \widetilde m}_\Domain^2\dt\\
\frac{1}{T}\int_0^T \Norm{n}_\bbP^2 &\lesssim \frac{1}{T}\int_0^T  \Norm{\widetilde m}_{\bbP^*}^2 \dt + T\int_0^T\Norm{\partial_t \widetilde m + \calL \widetilde p}_{\bbP^*}^2\dt\\
\Norm{n_0}^2_\bbP &\lesssim \Norm{\widetilde m(0)}^2_{\bbP^*}.
\end{alignedat}
\end{equation*} 
Notice that the hidden constants depends only on $\Domain$ through $c$ and $C$ in \eqref{E:inf-sup-B}. Estimating
the $L^2(\bbP^*)$-norm of $\widetilde m$ in terms of the
$L^\infty(\bbP^*)$-norm yields the desired bound 
\begin{equation*}
\Norm{y_{2,s}}^2_2 \lesssim \left( \Norm{\widetilde y_1}^2_1 + \Norm{\widetilde m}^2_{L^\infty(\bbP^*)} + \lambda\int_0^T \Norm{\calD \widetilde u}^2_\Domain\right).
\end{equation*}
We conclude that \eqref{E:inf-sup-proof-upper-bound} holds true by combining this with \eqref{E:inf-sup-test-function}.

\subsection{Nondegeneracy}
\label{SS:Nondegeneracy}

This section is devoted to the proof of
Lemma~\ref{L:nondegeneracy}. To this end, let $y_2 = (v, q_\Tot, q, n,
n_0) \in \bbY_2$ be such that \eqref{E:nondegeneracy} holds true for
all $\widetilde y_1 \in \overline{\bbY}_1$. We aim at showing  
\begin{equation}
\label{E:nondegeneracy-goal}
y_2 = 0.
\end{equation}

First, we use trial functions in the form  $\widetilde y_1 = (0, \widetilde
p_\Tot, 0, 0)$ in \eqref{E:nondegeneracy}. We obtain
\begin{equation*}
\int_0^T ( \widetilde p_\Tot,\, \calD v - q_\Tot) _\Domain \dt = 0
\end{equation*} 
for all $\widetilde p_\Tot \in L^2(\bbD)$. Since both $\calD v$ and $q_\Tot$ are in $L^2(\bbD)$, we infer 
\begin{equation}
\label{E:nondegeneracy-identity0}
q_\Tot = \calD v \qquad \text{in $L^2(\bbD)$}.
\end{equation}

Second, we use trial functions in the form $\widetilde y_1 = (\widetilde u, 0, 0, 0)$ in \eqref{E:nondegeneracy}. We obtain 
\begin{equation*}
\int_{0}^T\left( \left\langle \calE \widetilde u,\, v \right\rangle_\bbU
  + \lambda ( \calD \widetilde u,\, q_\Tot)_\Domain + \alpha ( \calD
  \widetilde u,\, \calP_\bbD q)_\Domain  \right) \dt = 0 
\end{equation*}
for all $u \in L^2(\bbU)$. We rearrange terms and use \eqref{E:nondegeneracy-identity0}. It results 
\begin{equation*}
\int_{0}^T \left\langle \widetilde u,\, \calE v + \lambda \calD^* \calD v + \alpha \calD^* \calP_\bbD q\right\rangle_\bbU \dt = 0. 
\end{equation*}
The operator $\calQ: \bbU \to \bbU^*$ defined as $\calQ = \calE +
\lambda \calD^*\calD$ is the one involved in the displacement
formulation of the linear elasticity equations. In particular, it is
self-adjoint and invertible. Since both $\calQ v$ and $\calD^*
\calP_\bbD q$ are in $L^2(\bbU^*)$, we infer 
\begin{equation}
\label{E:nondegeneracy-identity1}
v = -\alpha \calQ^{-1} \calD^* \calP_\bbD q \qquad \text{in $L^2(\bbU)$}.
\end{equation}

Third, we use trial functions in the form $\widetilde y_1 = (0, 0, \widetilde
p, 0)$ in \eqref{E:nondegeneracy} to obtain 
\begin{equation*}
\int_{0}^T \left( 
-\alpha ( \widetilde p,\, \calP_{\overline \bbP} q_\Tot)_\Domain +
\sigma ( \widetilde p,\, q)_\Domain + \left\langle \calL \widetilde p,\, n
\right\rangle_{\bbP} \right) \dt = 0  
\end{equation*}for all
$\widetilde p \in L^2(\bbP)$. Rearranging terms and using
\eqref{E:nondegeneracy-identity0} and
\eqref{E:nondegeneracy-identity1} results in
\begin{equation*}
\int_{0}^T ( \widetilde p,\, \alpha^2 \calP_{\overline \bbP}\calD
\calQ^{-1} \calD^*\calP_\bbD q + \sigma q )_\Domain \dt = -\int_{0}^T
\left\langle\widetilde p,\, \calL n \right\rangle_\bbP \dt. 
\end{equation*} 
Recall that $\bbP \subseteq \overline{\bbP} \equiv \overline{\bbP}^*
\subseteq \bbP^*$ is a Hilbert triplet. Since both $\calP_{\overline
  \bbP}\calD \calQ^{-1} \calD^*\calP_\bbD q$ and $q$ are in
$L^2(\overline{\bbP})$, we infer the inclusion $\calL n \in
L^2(\overline{\bbP})$ with  
\begin{equation}
\label{E:nondegeneracy-identity2}
\calL n = -\alpha^2 \calP_{\overline \bbP}\calD \calQ^{-1} \calD^*\calP_\bbD q - \sigma q \qquad \text{in $L^2(\overline{\bbP})$}.
\end{equation}

Finally, we use trial functions in the form $\widetilde y_1 = (0, 0, 0, \widetilde m)$ in \eqref{E:nondegeneracy}. We obtain 
\begin{equation*}
\int_{0}^T\Big( -( \widetilde m,\, q)_\Domain + \left\langle \partial_t
  \widetilde m,\, n \right\rangle_\bbP \Big)\dt + \left\langle
  \widetilde m(0),\, n_0\right\rangle_\bbP = 0 
\end{equation*}  
for all $\widetilde m \in L^2(\overline{\bbP}) \cap
H^1(\bbP^*)$. Considering first $\widetilde m = \phi \widetilde w$
with $\phi \in C^\infty_0(0,T)$ and $\widetilde w \in \overline{\bbP}$
reveals
\begin{equation*}
\int_{0}^T \phi ( \widetilde w,\, q  )_\Domain \dt =  \int_{0}^T \partial_t \phi \left\langle \widetilde w,\, n\right\rangle_\bbP \dt.
\end{equation*} 
By \cite[Proposition~64.33]{Ern.Guermond:21c}, we infer $n \in L^2(\bbP) \cap H^1(\overline{\bbP})$ with 
\begin{equation}
\label{E:nondegeneracy-identity3}
\partial_t n = -q \qquad \text{in $L^2(\overline{\bbP})$}.
\end{equation}
Then, assuming $\phi \in C^\infty(0, T)$ with $\phi(0) = 1$ and
$\phi(T) = 0$ or, respectively, $\phi(1) = 0$ and $\phi(T) = 1$, it
follows that $n(0), n(T) \in \bbP$ with   
\begin{equation}
\label{E:nondegeneracy-identity4}
n(0) = n_0 \quad \text{and}\quad  n(T) = 0\qquad \text{in $\bbP$}.
\end{equation}

According to \eqref{E:nondegeneracy-identity2}, \eqref{E:nondegeneracy-identity3} and \eqref{E:nondegeneracy-identity4}, it holds that
\begin{equation*}
\begin{split}
-\dfrac{1}{2} \Norm{n_0}^2_\bbP &= \dfrac{1}{2} \Norm{n(T)}^2_\bbP -
\dfrac{1}{2} \Norm{n(0)}^2_\bbP = \int_{0}^T ( \partial_t n,\, \calL
n)_\Domain \dt
\\
&=\alpha^2 \int_{0}^T \langle  \calQ^{-1} \calD^* \calP_\bbD q,\, \calD^* \calP_\bbD q \rangle _\bbU \dt + \sigma \int_{0}^T \Norm{q}_\Domain^2 \dt. 
\end{split}
\end{equation*}
Therefore, we have $n_0 = 0$, $\calP_\bbD q = 0$ and $\sigma q = 0$ in
the respective spaces. This implies $q = 0$ because we have
$\overline{\bbP} \subseteq \bbD$ (hence $\calP_\bbD q = q$) for
$\sigma = 0$, cf. \eqref{E:abstract-setting-spaces-ptot-m}. Then, the
identities \eqref{E:nondegeneracy-identity1},
\eqref{E:nondegeneracy-identity2} and
\eqref{E:nondegeneracy-identity0} imply $v = 0$, $n = 0$ and $q_\Tot
=0$, respectively. This verifies \eqref{E:nondegeneracy-goal} and
completes the proof.

\section{Further stability estimates}
\label{S:Stability-Additional}

The definition of the trial norm $\Norm{\cdot}_1$ in
\eqref{E:trial-norm} is, in a sense, minimal, because it includes only
the terms that are necessary for an immediate proof of the boundedness
of the form $b$ in \eqref{E:BiotProblem-abstract-form},
cf. Lemma~\ref{L:boundedness}. Lemma~\ref{L:inf-sup} reveals that we
could equivalently augment $\Norm{\cdot}_1$ with other terms, namely
the $L^2(L^2(\Domain))$-norm of the solid dilation and of the pressure
as well as the $L^\infty(\bbP^*)$-norm of the total fluid
content. This section aims at determining if we can control other
relevant terms by the trial norm. In particular, we are interested in
understanding whether $\bbY_1$ is closed with respect to
$\Norm{\cdot}_1$. Indeed, for $\widetilde y_1 = (\widetilde u,
\widetilde p_\Tot, 
\widetilde p, \widetilde m) \in \overline{\bbY}_1$, we have   
\begin{equation*}
\Norm{\widetilde y_1}_{1}^2 \geq T\int_0^T \Norm{\partial_t \widetilde m +
	\calL \widetilde p}^2_{\bbP^*}\dt  
\end{equation*}
but it is not clear whether $\Norm{\widetilde y_1}_1^2$  controls $\int_0^T(\Norm{\partial_t \widetilde m}_{\bbP^*}
+ \Norm{\widetilde p}_{\bbP}^2)\dt$. Thus, the question
arises if $\overline{\bbY}_1$ is indeed larger than $\bbY_1$ and, if
yes, how much larger it is. We establish some
results, showing that the two spaces are actually different but the difference is somehow subtle.

\subsection{Positive results}
\label{SS:positive-results}
We first show that we can control the $L^2(\bbP)$-norm of
an approximation of the the third component
$\widetilde p$ of a trial function, namely the
$L^2$-orthogonal projection onto $\bbP$-valued polynomials of any
degree $r \in \N_0$. Unfortunately, the control is not uniform with respect to $r$. 

The space of $\bbP$-valued polynomials of degree $r$ is defined as  
\begin{equation*}
\label{E:polynomials}
\Poly_r(\bbP) := \Big\{  q \in L^2(\bbP) \mid q(t) =
  \sum_{j=0}^{r} w_j t^j \quad \text{with} \quad (w_j)_{j=0}^r
  \subseteq \bbP \Big\}. 
\end{equation*}
The $L^2(\bbP)$-orthogonal projection $\calP_r: L^2(\bbP) \to \Poly_r(\bbP)$ is obtained via the condition
\begin{equation}
\label{E:projection-onto-polynomials}
\int_0^T \left\langle \calL \calP_r \widetilde p,\,q \right\rangle_\bbP \dt
= \int_0^T \left\langle \calL \widetilde p,\, q\right\rangle_\bbP \dt 
\end{equation}
for $\widetilde p \in L^2(\bbP)$ and for all $q \in \Poly_r(\bbP)$. Indeed, recall
that $\left\langle \calL \cdot,\, \cdot \right\rangle_\bbP$ is the
scalar product inducing the norm $\Norm{\cdot}_\bbP$ on $\bbP$,
cf. \eqref{E:coercivity-E-A}.  

The next proposition shows that the mapping
\begin{equation*}
(\widetilde u, \widetilde p_\Tot, \widetilde p, \widetilde m) \mapsto \calP_r \widetilde p
\end{equation*}
defines a bounded linear operator on $\bbY_1$ with respect
to the norm $\Norm{\cdot}_1$. Hence, we can extend such operator to
the trial space $\overline{\bbY}_1$ by density. In other words, the
$L^2(\bbP)$-orthogonal projection of the fluid pressure $\widetilde p$ onto
$\Poly_r(\bbP)$ is well-defined and bounded on $\overline{\bbY}_1$.

\begin{proposition}[Polynomial-in-time projection of the pressure]
\label{P:projection-fluid-pressure}
Let $r \in \N_0$. For $\widetilde y_1 = (\widetilde u, \widetilde p_\Tot, \widetilde p, \widetilde m) \in
\bbY_1$, it holds that 
\begin{equation*}
\label{E:projection-fluid-pressure}
T\int_0^T \Norm{\calP_r \widetilde p}_{\bbP}^2 \dt  
\lesssim 
\Norm{\widetilde y_1}_1^2.
\end{equation*}
The hidden constant depends only on the domain $\Domain$ and the degree~$r$.
\end{proposition}

\begin{proof}
Recall the boundedness of the form $b$ stated in
Lemma~\ref{L:boundedness}. Using the test function $y_2 = (0, 0, 0,
\calP_r \widetilde p, 0) \in \bbY_2$ reveals 
\begin{equation}
\label{E:projection-fluid-pressure-proof}
\int_0^T \left\langle \partial_t \widetilde m + \calL \widetilde p, \calP_r
  \widetilde p \right\rangle_\bbP \dt \lesssim T^{-\frac{1}{2}}\Norm{\widetilde y_1}_1 
   \Norm{\calP_r \widetilde p}_{L^2(\bbP)} 
\end{equation}
where the hidden constant depends only on $\Domain$. The inclusion
$\Poly_r(\bbP) \subseteq H^1(\bbP)$, the second identity in
\eqref{E:coercivity-E-A} and \eqref{E:projection-onto-polynomials}
entail that we have 
\begin{equation*}
\begin{split}
\int_0^T \left\langle \partial_t \widetilde m + \calL \widetilde p, \calP_r \widetilde p \right\rangle_\bbP \dt
= &\Norm{\calP_r \widetilde p}_{L^2(\bbP)}^2  -\int_0^T \left\langle \widetilde m, \partial_t \calP_r \widetilde p \right\rangle_\bbP \dt\\ 
&+ \left\langle \widetilde m(T), \calP_r \widetilde p(T) \right\rangle_\bbP - \left\langle \widetilde m(0), \calP_r \widetilde p(0) \right\rangle_\bbP.
\end{split}  
\end{equation*}
Note that the space $\Poly_r(\bbP)$ is complete with respect to both the $L^\infty(\bbP)$- and the $L^2(\bbP)$-norm. Therefore, the open mapping theorem implies the inverse inequality $\Norm{\calP_r \widetilde p}_{L^\infty(\bbP)}\lesssim T^{-\frac12}\Norm{\calP_r \widetilde p}_{L^2(\bbP)}$ and thus
\begin{equation*}
\left\langle \widetilde m(T), \calP_r \widetilde p(T) \right\rangle_\bbP - \left\langle \widetilde m(0), \calP_r \widetilde p(0) \right\rangle_\bbP \lesssim
T^{-\frac{1}{2}}\Norm{\widetilde m}_{L^\infty(\bbP^*)} \Norm{\calP_r \widetilde p}_{L^2(\bbP)} \dt
\end{equation*}
with hidden constants depending on $\Poly_r(\bbP)$ itself, hence on $\Domain$ and $r$. A similar argument entails  the inverse inequality $\Norm{\partial_t\calP_r \widetilde p}_{L^2(\bbP)}\lesssim T^{-1} \Norm{\calP_r \widetilde p}_{L^2(\bbP)}$, and we obtain with a H\"older estimate that
\begin{align*}
\int_0^T \left\langle \widetilde m, \partial_t \calP_r \widetilde p \right\rangle_\bbP \dt 
  \le T^{\frac{1}{2}}\Norm{\widetilde m}_{L^\infty(\bbP^*)}
  \Norm{\partial_t\calP_r \widetilde p}_{L^2(\bbP)}
  \lesssim
T^{-\frac{1}{2}}\Norm{\widetilde m}_{L^\infty(\bbP^*)} \Norm{\calP_r \widetilde p}_{L^2(\bbP)}. 
\end{align*}
We combine these bounds and the previous identity with
\eqref{E:projection-fluid-pressure-proof} to arrive at
\begin{equation*}
T\int_0^T \Norm{\calP_r \widetilde p}^2_\bbP \dt
\lesssim
\Norm{\widetilde y_1}^2_1 + \Norm{\widetilde m}^2_{L^\infty(\bbP^*)}.
\end{equation*}
We conclude by recalling the estimate $\Norm{\widetilde m}_{L^\infty(\bbP^*)} \lesssim \Norm{\widetilde y_1}_1$ from \eqref{E:continuity-m-bound}, with the hidden constant depending only on $\Domain$.
\end{proof}

Second, the estimate \eqref{E:continuity-m-bound} controls also the $L^\infty(\bbP^*)$-norm of $\widetilde m$, that is the anti-derivative of $\partial_t \widetilde m$. This readily implies that we can control the $L^\infty(\bbP)$-norm of the anti-derivative of $\widetilde p$. 

\begin{proposition}[Anti-derivative of the pressure]
\label{P:contr-anti-deriv}
For
$\widetilde y_1 = (\widetilde u, \widetilde p_\Tot, \widetilde p, \widetilde m) \in
\bbY_1$, we have
\begin{equation*}
\label{E:antiderivative-fluid-pressure}
\min\{1,T\}\sup_{t\in [0,T]} \left \|\int_0^t \widetilde p\dt\right \|_{\bbP}^2 
\lesssim 
\Norm{\widetilde y_1}^2_1.
\end{equation*}
The hidden constant depends only on the domain $\Domain$.
\end{proposition}

\begin{proof}
  We observe first that \(\int_0^t\partial_t \widetilde m\dt=\widetilde
  m(t)-\widetilde m(0)\) on
  \(\bbP^*\). Therefore, by applying a triangle and a H\"older
  inequality and recalling the definition of $\Norm{\cdot}_1$, we obtain
  \begin{align*}
    \left \|\int_0^t \widetilde p \dt\right \|_{\bbP}^2 &\lesssim \left \|\int_0^t \Big (
    \partial_t \widetilde m+\calL\widetilde p\Big) \;\dt \right \|_{\bbP^*}^2 +
    \Norm{\widetilde m(t)}_{\bbP^*}^2+\Norm{\widetilde m(0)}_{\bbP^*}^2
    \\
    &\lesssim T^{-1}\Norm{\widetilde y_1}_1^2 +
    \Norm{\widetilde m}_{L^\infty(\bbP^*)}^2.
  \end{align*}
  The claimed estimate follows from \eqref{E:continuity-m-bound}.
\end{proof}

\subsection{Negative results}
\label{SS:negative -results}

Roughly speaking, the results in Section~\ref{SS:positive-results} state
that, if the space $\bbY_1$ differs from its closure
$\overline{\bbY}_1$, then the difference is somehow subtle. On the
other hand, the next result confirms that indeed the two spaces are
different. In view of Proposition~\ref{P:projection-fluid-pressure},
the proof builds upon the construction of a function in
$\overline{\bbY}_1$, so that the third and the fourth components are polynomials in time of arbitrarily high degree.  

\begin{proposition}[Existence of `rough' trial functions]
\label{P:rough-trial-functions}
It holds that
\begin{equation*}
\label{E:rough-trial-functions}
\sup_{\widetilde y_1 = (\widetilde u, \widetilde p_\Tot, \widetilde p, \widetilde m) \in \bbY_1} 
\dfrac{\displaystyle\int_0^T \Big( \Norm{\partial_t \widetilde m}^2_{\bbP^*} + \Norm{\widetilde p}^2_\bbP \Big)\dt}{\Norm{\widetilde y_1}_1^2} = +\infty.
\end{equation*} 
\end{proposition}

\begin{proof}
Recall that $\bbP \subseteq \overline{\bbP} \equiv \overline{\bbP}^* \subseteq\bbP^*$ is a Hilbert triplet. Then, by the theory of self-adjoint coercive operators, we have, for the eigenvalues $(\lambda_k)_{k \geq 1} \subseteq (0, +\infty)$ of the operator $\calL$, that $\lambda_k \nearrow +\infty$ as $k \to +\infty$. Let $(w_k)_{k \geq 1} \subseteq \bbP$ be the associated eigenfunctions. Hence, we have $\calL w_k = \lambda_k w_k$ as well as 
\begin{equation}
\label{E:eigenfunctions-norm}
\Norm{w_k}_\Domain^2 = 1
\qquad \text{and} \qquad
\Norm{w_k}^2_\bbP = \lambda_k, \qquad k \geq 1.
\end{equation}

Denote by $r_k := \left \lceil \lambda_k \right \rceil$ the ceiling function
applied to $\lambda_k$, i.e., the smallest integer larger than or equal to $\lambda_k$. We define
$\widetilde p^{(k)} \in L^2(\bbP)$ and $\widetilde m^{(k)} \in 
L^2(\overline{\bbP}) \cap H^1(\bbP^*)$ by
\begin{equation*}
\widetilde p^{(k)}(t) := \dfrac{w_k }{T^{r_k}} t^{r_k} 
\qquad \text{and} \qquad
\widetilde m^{(k)}(t) := - \dfrac{\lambda_k w_k }{T^{r_k}(r_k+1)} t^{r_k + 1}. 
\end{equation*}
By construction, we have
\begin{equation}
\label{E:rough-trial-function-construction}
\partial_t \widetilde m^{(k)} + \calL \widetilde p^{(k)} = 0
\qquad \text{and} \qquad
\widetilde m^{(k)} (0) =  0.
\end{equation}
Moreover, elementary computations reveal
\begin{equation*}
\int_0^T \Norm{\widetilde{p}^{(k)}}^2_\Domain \dt
=
\dfrac{T}{2r_k +1}
\qquad \text{and} \qquad
\int_0^T \Norm{\widetilde m^{(k)}}^2_\Domain\dt = \dfrac{T^3 \lambda_k^2}{(r_k+1)^2(2r_k+3)}.
\end{equation*}
Thus, for $\widetilde y_1^{(k)} = (0, 0, \widetilde p^{(k)}, \widetilde m^{(k)}) \in \bbY_1$, the definition \eqref{E:trial-norm} of the trial norm implies
\begin{equation*}
\Norm{\widetilde y_1^{(k)}}_1^2 \to 0 \quad \text{as $k \to +\infty$}.
\end{equation*}
On the other hand, it holds that 
\begin{equation*}
\int_0^T \Norm{\partial_t \widetilde m^{(k)}}^2_{\bbP^*}
=
\int_0^T \Norm{\widetilde p^{(k)}}^2_\bbP
=
\dfrac{\lambda_k T}{2r_k +1} 
\to \dfrac{T}{2} \quad \text{as $k \to +\infty$}.
\end{equation*}
Comparing this limit with the previous one concludes the proof.
\end{proof}

\begin{remark}[`Rough' trial functions]
\label{R:rough-trial-function}
The elements of the sequence $(\widetilde y_1^{(k)})_{k\geq 1}$ in the
poof of Proposition~\ref{P:rough-trial-functions} are in
$\bbY_1$ but do not solve the weak formulation
\eqref{E:BiotProblem-weak-formulation} of the Biot's
equations. Indeed, they do not fulfill the constraints in the second
and in the third lines of
\eqref{E:BiotProblem-weak-formulation}. Still, we might modify the
first component of $\widetilde y_1$ by choosing it (more precisely,
its divergence) so as to fulfill the constraint in the third
line. Then, we might modify also the second component according to the
second line in \eqref{E:BiotProblem-weak-formulation}. This
observation reveals a remarkable difference between our analysis and
former ones: we do not control the $L^2(\bbP)$-norm of the pressure, due to the weaker regularity assumption on the load $\ell_u$ in \eqref{E:BiotProblem-weak-formulation}, cf. Remark~\ref{R:regularity-lu}.  
\end{remark}

\begin{remark}[Time regularity]
\label{R:comparison-Murad-Tomee-Loula}
Some results in the spirit of
Propositions~\ref{P:contr-anti-deriv}-\ref{P:rough-trial-functions}
are proved by Murad, Thom\'{e}e and Loula
\cite[Section~2]{Murad.Thomee.Loula:96} under the assumption $\ell_u \in
H^1(\bbU^*)$ in \eqref{E:BiotProblem-weak-formulation}. Indeed, a
bound on the $L^\infty(\bbP)$-norm of the anti-derivative of the
pressure is established and it is observed that the same bound does
not hold true for the pressure itself, because of a singularity at $ t
= 0$. Numerical evidence of the latter observation can be found also in \cite{Bociu.Guidoboni.Sacco.Webster:16}. As in Remark~\ref{R:comparison-Zenisek}, the higher
integrability in time, compared to our approach, follows from the
higher regularity of $\ell_u$, cf. Remark~\ref{R:regularity-lu}.
\end{remark}

We conclude this section by recalling that, in the analysis of scalar
parabolic equations, the $L^\infty$ control in time over the point values of
the solution is obtained by combining some $L^2$ control over the
solution itself and on its time derivative. Here, in contrast,
Proposition~\ref{P:rough-trial-functions} states that we do not have a
$L^2$ control over the time derivative of the total fluid
content. Hence, it is remarkable that we could nevertheless establish the embedding
of such variable into $C^0(\bbP^*)$, cf. Remark~\ref{R:continuity-m}.   

\section{Shift of the regularity in space}
\label{S:Regularity}

In addition to the well-posedness of the weak formulation
\eqref{E:BiotProblem-weak-formulation}, we are interested also in
shift theorems, i.e. the question, whether more regular data than in
Theorem~\ref{T:well-posedness} give rise to correspondingly more
regular solutions. In fact, the regularity of the solution is a
necessary ingredient to justify the error decay for the discretization we propose and analyze in \cite{Kreuzer.Zanotti:22+}. Still, it is known that the
regularity theory for the Biot's equations is subtle. For instance,
Murad, Thom\'{e}e and Loula \cite{Murad.Thomee.Loula:96} observed that
the regularity in time is limited at $t=0$ and also Showalter
\cite{Showalter:00} came to a similar conclusion. 

Due to the complexity of the subject, we do not attempt at establishing a comprehensive result here. We confine our discussion to a rather specific case, where no singularity occurs. More precisely, we investigate only the regularity in space under the set of assumptions detailed below. What is more relevant for us is that we can use the inf-sup theory once more to this end. This appears to be an innovative technique not only for the Biot's equations but also in the general framework of coupled problems. In fact, we introduce another variational formulation of the initial-boundary value problem \eqref{E:BiotProblem} for the Biot's equations. Compared to the weak formulation \eqref{E:BiotProblem-weak-formulation}, we prescribe additional regularity in space of the trial functions. Therefore, any solution of the new `strong' formulation solves also the weak one. We verify the well-posedness again by applying the Banach-Ne\u{c}as-Babu\v{s}ka theorem. In this way, we establish a two-sided estimate, in the vein of \eqref{E:well-posedness}, ensuring that the regularity guaranteed for the solution matches the regularity requirement for the data.

\begin{subequations}
	\label{E:regularity-assumptions}
Our first assumption concerns the domain $\Domain$. We require
\begin{equation}
\label{E:regularity-assumption-domain}
\Domain \subseteq \R^2 \,\text{ is a convex polygon}.
\end{equation}
Alternatively, we could work with $\partial \Domain$ smooth, but \eqref{E:regularity-assumption-domain} is more relevant for the discretization analyzed in \cite{Kreuzer.Zanotti:22+}. Second, we assume pure essential boundary conditions for the displacement and pure natural boundary conditions for the pressure
\begin{equation}
\label{E:regularity-assumption-BCs}
\Gamma_{u,E} = \partial \Domain = \Gamma_{p,N}.
\end{equation}
Third, we restrict ourselves to the usually more critical case for the Lam\'{e} constants 
\begin{equation}
\label{E:regularity-assumption-parameters}
\Lambda \mu \leq \lambda 
\end{equation}
for some constant $\Lambda$ that is as large as necessary for the arguments in sections~\ref{SS:InfSup-Strong} and \ref{SS:Nondegeneracy-strong}, i.e. $\mu \ll \lambda$. Note that the chosen setting implies $\overline{\bbP}=\bbD$ (see~\eqref{E:abstract-setting-spaces-ptot-m}) and thus $\gamma$ in~\eqref{E:gamma} remains bounded even for  $\sigma \searrow 0$.
\end{subequations} 

\begin{remark}[Justification of the assumptions]
\label{R:regularity-assumption-justification}
Linear elasticity is one of the building blocks in the Biot's equations. Therefore, we use arguments introduced by Brenner and Sung \cite[Section~2]{Brenner.Sung:92} for that problem. This motivates the assumptions \eqref{E:regularity-assumption-domain} and \eqref{E:regularity-assumption-parameters}, as well as the first part of \eqref{E:regularity-assumption-BCs}. Note that \cite{Brenner.Sung:92} covers also pure natural boundary conditions. Regarding the restriction to $d=2$, the critical result is~\cite[Lemma~2.1]{Brenner.Sung:92}, which we use in~\eqref{E:infsup-strong-rightinverse} below.
The second part of \eqref{E:regularity-assumption-BCs} implies that we can use the same space for the total pressure and the total fluid content (cf. Figure~\ref{F:abstract-spaces-diagram-strong}), and an important relation between the differential operators, see \eqref{E:abstract-operator-laplacian-gradient}. Our proof heavily exploits both properties.
\end{remark}

\subsection{Strong formulation in space}
\label{SS:strong-formulation}
First of all, we introduce dedicated symbols for the most frequently used spaces, in the vein of Section~\ref{SS:weak-formulation}. According to the assumptions \eqref{E:regularity-assumptions}, we use
\begin{equation*}
\label{E:regularity-space-displacement}
\bbE := H^2(\Domain)^2 \cap \bbU = H^2(\Domain)^2 \cap H^1_0(\Domain)^2
\end{equation*}
for the regularity in space of the displacement. For the total pressure and the total fluid content, we recall the space $\bbP$, which reads
\begin{equation*}
\label{E:regularity-space-auxiliary-variables}
\bbP = H^1(\Domain) \cap L^2_0(\Domain)
\end{equation*}
in this case. For the pressure, we recall the space $\mathbb{L}$ from Remark~\ref{R:comparison-Li-Zikatanov}, namely
\begin{equation*}
\label{E:regularity-space-pressure}
\bbL = \{ \widetilde p \in \bbP \mid \mathcal{L} \widetilde p \in \overline \bbP \}
= \{ \widetilde p \in \bbP \mid \mathcal{L} \widetilde p \in L^2_0(\Domain) \}
\end{equation*}
where the closure of $\bbP$ is taken with respect to the $L^2(\Domain)$-norm.

The interplay between these spaces and (the restriction of) the  differential operators in \eqref{E:abstract-operators-concrete} is summarized in Figure~\ref{F:abstract-spaces-diagram-strong}. Note the different structure compared to the weak formulation (Figure~\ref{F:abstract-spaces-diagram}) and that, in the left part of the diagram, $\bbP$ and $L^2(\Domain)^2$ are identified with subspaces of $\bbD^*$ and $\bbU^*$ via the $L^2(\Domain)$-scalar product. Upon this identification, we have
\begin{equation}
\label{E:abstract-operators-gradient}
\calD^* = -\Grad \quad \text{ in $\bbP$}
\end{equation}
as well as
\begin{equation}
\label{E:abstract-operator-laplacian-gradient}
\left\langle \calL \cdot, \,\cdot\right\rangle_\bbP =  \kappa (\calD^* \cdot,\, \calD^* \cdot)_\Domain \quad \text{ in $\bbP \times \bbP$}.
\end{equation}

\begin{figure}[ht]
	\[
	\xymatrixcolsep{4pc}
	\xymatrixrowsep{4pc}
	\xymatrix{
		\bbE 
		\ar[dr]^{\calD}
		\ar[d]^{\calE}
		&
		&
		\bbL
		\ar[dl]_{i}
		\ar[d]^{\calL}
		\\
		L^2(\Domain)^2
		&
		\bbP
		\ar[l]^{\calD^*}
		\ar[r]_{i}
		&
		\overline{\bbP}
	}
	\]
	\caption{\label{F:abstract-spaces-diagram-strong} Spaces and
          operators describing the regularity in space for the strong
          formulation \eqref{E:BiotProblem-strong-formulation} of the
          Biot's equations. Note that $i$ denotes the embedding
          operator.} 
\end{figure}
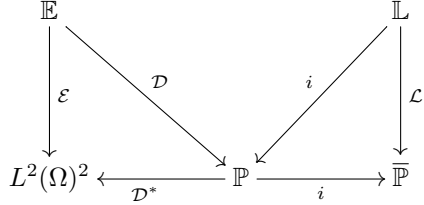

Following \cite[Section~2.1]{Boffi.Brezzi.Fortin:13}, we define
$H^{\frac{1}{2}}(\partial \Domain)$ as the image of $\bbP$ via the
trace operator on $\partial \Domain$. We equip this space with the
quotient norm 
\begin{equation}
\label{E:norm-1/2}
\| g \|_{H^{\frac{1}{2}}(\partial \Domain)}^2 :=
\dfrac{1}{\kappa}\inf_{\widetilde p \in \bbP, \;\widetilde
  p_{|\partial \Domain} = g} \Norm{\widetilde p}_\bbP^2 =
\inf_{\widetilde p \in \bbP, \;\widetilde p_{|\partial \Domain} = g}
\Norm{\calD^* \widetilde p}_\Domain^2. 
\end{equation}
Let $H^{-\frac{1}{2}}(\partial \Domain)$ be the dual of
$H^{\frac{1}{2}}(\partial \Domain)$. For simplicity, we denote by
$\left\langle \cdot, \,\cdot \right\rangle_{1/2} $ the corresponding
duality. By comparing \eqref{E:abstract-operators-gradient} with the
definitions \eqref{E:abstract-operators-concrete} and
\eqref{E:regularity-space-pressure} of $\calL$ and $\bbL$,
respectively, we infer that the normal derivative $\partial_\Normal:
\bbL \to H^{-\frac{1}{2}}(\partial \Domain)$ is well-defined and we
have 
\begin{equation}
\label{E:normal-derivative}
\left\langle \calL \cdot,\, \cdot\right\rangle_\bbP = (\calL  \cdot,\,
\cdot)_\Domain + \kappa \left\langle \partial_\Normal \cdot,\,
  \cdot\right\rangle_{1/2}  \quad \text{ in $\bbL \times \bbP$}. 
\end{equation}

After this preparation, we are in position to state the announced
strong formulation of \eqref{E:BiotProblem-equations-fourfield}, with
\eqref{E:BiotProblem-initialcondition} and \eqref{E:Biot-BCs-mixed},
as follows 
\begin{equation}
\label{E:BiotProblem-strong-formulation}
\begin{alignedat}{2}
\calE u + \calD^* p_\Tot &= f_u \qquad && \text{in $L^2(L^2(\Domain)^2)$}\\
\lambda \calD u - p_\Tot - \alpha p &= 0 \qquad && \text{in $L^2(\bbP)$}\\
\alpha \calD u + \sigma p - m &= 0 \qquad && \text{in $L^2(\bbP)$}\\
\partial_t m + \calL p &= f_p \qquad  && \text{in $L^2(\overline{\bbP})$}\\
\partial_\Normal p &= g_p \qquad && \text{in $L^2(H^{-\frac{1}{2}}(\partial \Domain))$}\\
m(0) &= \ell_0 \qquad && \text{in $\overline{\bbP}$}.
\end{alignedat}
\end{equation}

Note that, compared to the weak formulation
\eqref{E:BiotProblem-weak-formulation}, the data are assumed to be
more regular in space and the boundary condition for the pressure is
satisfied in the sense of traces. We look for a solution of
\eqref{E:BiotProblem-strong-formulation} in the trial space $\overline
\bbX_1$, with 
\begin{equation*}
\label{E:trial-space-strong}
\bbX_1 := L^2(\bbE) \times L^2(\bbP) \times L^2(\bbL) \times \left( L^2(\bbP) \cap H^1(\overline{\bbP}) \right)
\end{equation*}
where the closure is taken with respect to the norm
\begin{equation*}
\label{E:trial-norm-strong}
\begin{split}
\Normtr{(\widetilde u, \widetilde p_\Tot, \widetilde p, \widetilde m)}_1^2 := 
&\int_{0}^{T} \left( \mu \Norm{D^2\widetilde u}^2_\Domain +
  \dfrac{1}{\mu}\Norm{\calD^* \widetilde p_\Tot}_\Domain^2 +
  \dfrac{T}{\kappa}\Norm{\partial_t \widetilde m + \calL \widetilde
    p}^2_{\Domain} \right)\dt
\\
+ &\int_{0}^T \left( \dfrac{1 }{\lambda} \Norm{\calD^*(\lambda \calD
	\widetilde u - \widetilde p_\Tot - \alpha \widetilde p)}^2_\Domain + \gamma\Norm{\calD^*(\alpha \calD \widetilde u + \sigma
	\widetilde p - \widetilde m)}^2_\Domain  \right)\dt
    \\
+ & \int_0^T \dfrac{1}{\gamma}\Norm{\partial_\Normal \widetilde
  p}^2_{H^{-\frac{1}{2}}(\partial \Domain)}\dt +
\dfrac{1}{\kappa}\Norm{\widetilde m(0)}^2_{\Domain} . 
\end{split}
\end{equation*}
According to the additional regularity of the data, each component of
$\bbX_1$ is more regular in space compared to the corresponding one of the space $\bbY_1$
for the weak formulation~\eqref{E:trial-space}.  

The corresponding test space is 
\begin{equation*}
\label{E:test-space-strong}
\bbX_2 := L^2(L^2(\Domain)^2) \times L^2(\bbP) \times L^2(\bbP) \times L^2(\overline{\bbP}) \times L^2(H^{\frac{1}{2}}(\partial \Domain)) \times \overline{\bbP}
\end{equation*}
equipped with the norm
\begin{equation}
\label{E:test-norm-strong}
\begin{split}
\Normtr{(v, q_\Tot, q, n, n_\partial, n_0)}_2^2 &:= \int_0^T \Big ( \mu \Norm{v}^2_\Domain +  \frac{\kappa}{T} \Norm{n}^2_{\Domain} + \gamma \Norm{n_\partial}_{H^{\frac{1}{2}}(\partial \Domain)}^2 \Big)\dt + \kappa\Norm{n_0}^2_\Domain\\
& + \int_{0}^T \Big(  \lambda \Norm{\calD^*q_\Tot}_\Domain^2 + \frac{1}{\gamma} \Norm{\calD^*q}^2_\Domain  \Big)\dt .
\end{split} 
\end{equation}

\begin{remark}[Closure of $\bbX_1$]
\label{R:closure-X1}
Like the trial space $\bbY_1$ for the weak formulation, the space $\bbX_1$ is not closed with respect to the norm defined on it. This can be verified by exactly the same argument as in the proof of Proposition~\ref{P:rough-trial-functions}. Indeed, the eigenfunctions of $\mathcal{L}$ are actually in $\bbL$. Note that, in analogy with the observation in Remark~\ref{R:comparison-Zenisek}, the $L^2(\bbL)$-norm of the pressure can be controlled upon assuming that the load $f_u$ in the first equation of \eqref{E:BiotProblem-strong-formulation} is weakly differentiable in time, cf. Remark~\ref{R:comparison-Li-Zikatanov}. This approach is widely used in connection with mixed formulations of the Biot's equations, introducing the Darcy velocity as an independent variable, see e.g. \cite{Li.Zikatanov:22}.
\end{remark}

Before discussing further properties of \eqref{E:BiotProblem-strong-formulation}, it is worth noticing that this is indeed a stronger formulation of the Biot's equations than \eqref{E:BiotProblem-weak-formulation}.

\begin{lemma}[Strong vs weak formulation]
\label{L:strong-vs-weak}
Assume $x_1 \in \overline \bbX_1$ solves \eqref{E:BiotProblem-strong-formulation} with data 
\begin{equation}
\label{E:data-strong}
f_u \in L^2(L^2(\Domain)^2), \qquad 
f_p \in L^2(\overline{\bbP}), \qquad
g_p \in L^2(H^{-\frac{1}{2}}(\partial \Domain)), \qquad
\ell_0 \in \overline{\bbP}.
\end{equation}
Then $x_1$ solves also \eqref{E:BiotProblem-weak-formulation} with the data defined by \eqref{E:abstract-data-concrete}.
\end{lemma}

\begin{proof}
The inclusion $\bbX_1 \subseteq \bbY_1$ and the bound $\Norm{\cdot}_1
\lesssim \Normtr{\cdot}_1$ imply $\overline{\bbX}_1 \subseteq
\overline \bbY_1$. Thus, if $x_1 = (u, p_\Tot, p, m) \in \overline
\bbX_1$ solves \eqref{E:BiotProblem-strong-formulation}, it is an
admissible trial function also for the weak formulation. The
assumptions \eqref{E:regularity-assumptions} entail that, in this
case, we have $\bbD = L^2_0(\Domain) = \overline \bbP$,
cf. \eqref{E:abstract-setting-spaces-ptot-m}. Then $x_1$ fulfills the
second, third and fifth equations in
\eqref{E:BiotProblem-weak-formulation}. The first equation is
fulfilled as well because of the inclusion $L^2(\Domain)^2 \subseteq
\bbU^*$ and the identity  
\begin{equation*}
\int_0^T(\calE u + \calD^* p_\Tot,\, v)_\Domain\dt = \int_0^T\left\langle \calE u + \calD^*p_\Tot,\, v\right\rangle_\bbU \dt
\end{equation*}
where $v \in L^2(\bbU)$ is arbitrary. Finally, for the fourth equation, we notice that \eqref{E:normal-derivative} implies
\begin{equation*}
\int_0^T \Big((\partial_t m + \calL p,\, n)_\Domain+ \kappa \left\langle \partial_\Normal p,\, n\right\rangle_{1/2} \Big)  \dt
=
\int_0^T \left\langle \partial_t m + \calL p,\, n\right\rangle_\bbP \dt  
\end{equation*}
for all $n \in L^2(\bbP)$.
\end{proof}

\subsection{Well-posedness and regularity}
\label{SS:regularity-results}

As mentioned before, the assumptions in \eqref{E:regularity-assumptions} prevent from any unexpected singularity of the solution of \eqref{E:BiotProblem-strong-formulation}. Therefore, the well-posedness can be verified by mimicking the argument in Section~\ref{S:WellPosedness}. Also in this case, we postpone most of the details of the proof to the next sections.

\begin{theorem}[Well-posedness]
\label{T:well-posedness-strong}
For $(f_u, f_p, g_p, \ell_0)$ as in \eqref{E:data-strong}, the equations \eqref{E:BiotProblem-strong-formulation} have a unique solution $x_1 = (u,p_\Tot, p, m) \in \overline{\bbX}_1$, which fulfills the two-sided stability bound
\begin{equation}
\label{E:well-posedness-strong}
\begin{split}
\Normtr{x_1}_1^2+&\int_0^T \Big( \lambda \Norm{\calD^*\calD u}^2_\Domain + \dfrac{1}{\gamma} \Norm{\calD^*p}^2_\Domain  \Big)  +  \dfrac{1}{\kappa}\Norm{m}^2_{L^\infty(\overline\bbP)}\\ 
& \eqsim \int_{0}^T \left( \dfrac{1}{\mu}\Norm{f_u}^2_{\Domain} + \dfrac{T}{\kappa}\Norm{f_p}^2_{\Domain} + \dfrac{1}{\gamma}\Norm{g_p}^2_{H^{-\frac{1}{2}}(\partial \Domain)} \right)\dt + \dfrac{1}{\kappa}\Norm{\ell_0}^2_{\Domain}.
\end{split}
\end{equation}
The hidden constants depend only on the domain $\Domain$ . Moreover, we have $m \in C^0(\overline \bbP)$.
\end{theorem}

\begin{proof}
The equations \eqref{E:BiotProblem-strong-formulation} are equivalent to a linear variational problem like \eqref{E:BiotProblem-variational} with the bilinear form $b: \overline \bbX_1 \times \bbX_2 \to \R$
\begin{equation}
\label{E:BiotProblem-abstract-form-strong}
\begin{split}
b(&\widetilde x_1, x_2) = \int_0^T \Big( ( \calE \widetilde u +
\calD^* \widetilde p_\Tot,\, v )_\Domain + (
\partial_t \widetilde m + \calL \widetilde p,\, n )_\Domain
\Big)\dt
\\
+ &\int_0^T \Big ( \left( \calD^*(\lambda \calD \widetilde u - \widetilde
p_\Tot - \alpha \widetilde p),\, \calD^*q_\Tot\right)_\Domain +
\left( \calD^*(\alpha \calD \widetilde u + \sigma \widetilde p - \widetilde m),\, \calD^*q
\right)_\Domain   \Big )\dt \\
+ &\int_0^T \left\langle \partial_\Normal \widetilde p,\, n_\partial\right\rangle_{1/2}   + ( \widetilde m(0),\, n_0 )_\Domain
\end{split} 
\end{equation}
and the load $\ell: \bbX_2 \to \R$
\begin{equation}
\label{E:BiotProblem-abstract-load-strong}
\ell(x_2) = \int_0^T \Big((f_u,\, v)_\Domain + (f_p,\, n)_\Domain + \left\langle g_p,\, n_\partial\right\rangle_{1/2}  \big)\dt + ( \ell_0,\, n_0 )_\Domain 
\end{equation}
for $\widetilde x_1 = (\widetilde u, \widetilde p_\Tot, \widetilde p, \widetilde m) \in \overline{\bbX}_1$ and $x_2 = (v, q_\Tot, q, n, n_\partial, n_0) \in \bbX_2 $. The boundedness of the bilinear form with respect to the norm $\Normtr{\cdot}_1$ and $\Normtr{\cdot}_2$ follows from Cauchy-Schwartz inequalities. Section~\ref{SS:InfSup-Strong} establishes the (strengthened) inf-sup stability
\begin{equation}
\label{E:inf-sup-strong}
\begin{split}
&\sup_{x_2 \in \bbX_2} \dfrac{b(\widetilde x_1, x_2)}{\Normtr{x_2}_2} \gtrsim\\
&\qquad\qquad\left(  \Normtr{\widetilde x_1}_1^2 +
\dfrac{1}{\kappa}\Norm{\widetilde m}^2_{L^\infty(\overline\bbP)} + \int_{0}^T \left( \lambda
\Norm{\calD^*\calD \widetilde u}^2_\Domain + \frac{1}{\gamma} \Norm{\calD^*\widetilde p
}^2_\Domain\right) \dt \right)^{\frac{1}{2}}  
\end{split}
\end{equation}
for $\widetilde x_1 \in \overline \bbX_1 $. Section \ref{SS:Nondegeneracy-strong} further verifies the nondegeneracy of $b$. The combination of these properties implies the well-posedness of the equations \eqref{E:BiotProblem-strong-formulation} by the Banach-Ne\u{c}as-Babu\v{s}ka theorem \cite[theorem~25.9]{Ern.Guermond:21b}. The estimate \eqref{T:well-posedness-strong} then follows by combining boundedness and inf-sup stability with the definitions \eqref{E:BiotProblem-abstract-load-strong} and \eqref{E:test-norm-strong} of the load and of the test norm $\Normtr{\cdot}_2$. Finally, the continuity in time of $m$ can be verified by arguing as in Remark~\ref{R:continuity-m}. 
\end{proof}

The combination of Theorem~\ref{T:well-posedness-strong} with Lemma~\ref{L:strong-vs-weak} implies the main result in this section, which establishes additional regularity in space of the solution of the weak formulation \eqref{E:BiotProblem-weak-formulation} with correspondingly more regular data.

\begin{corollary}[Regularity in space]
\label{C:regularity}
Let $(f_u, f_p, g_p, \ell_0)$ be as in \eqref{E:data-strong}. Denote
by $y_1 = (u, p_\Tot, p, m) \in \overline \bbY_1$ the  solution of the
equations \eqref{E:BiotProblem-weak-formulation} with the data
$\ell_u$ and $\ell_p$ defined by
\eqref{E:abstract-data-concrete}. Then we have $y_1 \in \overline
\bbX_1$ and $y_1$ fulfills \eqref{E:well-posedness-strong}. 
\end{corollary}

\begin{proof}
Owing to Theorem~\ref{T:well-posedness-strong}, the equations
\eqref{E:BiotProblem-strong-formulation}, with the given data, admit a
unique solution $x_1 \in \overline{\bbX}_1$. By
Lemma~\ref{L:strong-vs-weak}, $x_1$ solves also the equations
\eqref{E:BiotProblem-weak-formulation}. Then, we have $x_1 = y_1$
according to Theorem~\ref{T:well-posedness}. This confirms that $y_1$
is in $\overline{\bbX}_1$ and fulfills the estimate
\eqref{E:well-posedness-strong} in
Theorem~\ref{T:well-posedness-strong}. 
\end{proof}

\subsection{Inf-sup stability}
\label{SS:InfSup-Strong}

This section establishes the inf-sup stability
\eqref{E:inf-sup-strong} of the bilinear form $b$ in
\eqref{E:BiotProblem-abstract-form-strong}. For this purpose, let
$\widetilde x_1 = (\widetilde u, \widetilde p_\Tot, \widetilde p,
\widetilde m) \in \bbX_1$. We proceed analogously to
Section~\ref{SS:InfSup}, therefore, we mention only the main aspects of
the argument.  

Let $s \in [0, T]$ be arbitrary and denote by $\chi_s:[0,T] \to \R$
the indicator function on $[0, s]$. We consider the test function
$x_{2,s} = (v, q_\Tot, q, n, n_\partial, n_0) \in \bbX_2$ defined by 
\begin{equation}
\label{E:infsup-strong-test-function}
\begin{alignedat}{2}
v &= \dfrac{K_1}{\mu}\Big( \calE \widetilde u + \calD^* \widetilde
p_\Tot\Big)\chi_s \quad && \in L^2(L^2(\Omega)^2)
\\
q_\Tot &= \Big(\dfrac{2}{\lambda}\big(\lambda \calD \widetilde u -
\widetilde p_\Tot - \alpha \widetilde p\big) - \frac{\alpha K_2}{\lambda} \widetilde p \Big) \chi_s \quad &&\in
L^2(\mathbb{P})
\\
q &= \frac{5\gamma}{K_2} \Big( \alpha \calD \widetilde u + \sigma \widetilde p
- \widetilde m  \Big)\chi_s \quad && \in L^2(\mathbb{P})
\\
n &= \dfrac{1}{\kappa}\Big( 2 \widetilde m + s(\partial_t \widetilde m + \calL \widetilde
p )\Big)\chi_s \quad && \in L^2(\overline{\bbP})
\\
n_{\partial \Omega} &= \dfrac{7}{\gamma} \Big(\mathcal
R^{-1}_\partial \partial_\Normal \widetilde p\Big)\chi_s && \in
L^2(H^{-\frac{1}{2}}(\partial \Omega)) 
\\
n_0 &= \dfrac{2}{\kappa} \widetilde m(0) && \in \overline{\bbP}
\end{alignedat}
\end{equation} 
with the constants $K_1, K_2$ to be determined later and $\mathcal{R}_\partial:
H^{\frac{1}{2}}(\partial \Domain) \to H^{-\frac{1}{2}}(\partial
\Domain)$ denoting the Riesz isometry. 

Using this test function in~\eqref{E:BiotProblem-abstract-form-strong} yields
\begin{equation}
\label{E:infsup-strong-basic-estimate}
\begin{alignedat}{2}
b(\widetilde x_1, x_{2,s}) & = \dfrac{K_1}{\mu}\int_0^s \| \calE\widetilde u + \calD^* \widetilde p_\Tot \|^2_\Omega\dt \qquad \qquad && (=: \mathfrak{I}_1) \\
& + \dfrac{2}{\lambda}\int_0^s \|\calD^*( \lambda \calD \widetilde u - \widetilde p_\Tot - \alpha \widetilde p )\|_\Omega \dt&&\\
& - \dfrac{\alpha K_2}{\lambda}\int_0^s (\calD^*( \lambda \calD \widetilde u - \widetilde p_\Tot - \alpha \widetilde p ), \calD^* \widetilde p)_\Omega^2 \dt&\qquad&  (=: \mathfrak{I}_2)\\
& + \frac{5\gamma}{K_2} \int_0^s \| \calD^*(\alpha \calD \widetilde u + \sigma \widetilde p - \widetilde m) \|_\Omega^2 \dt&&\\
& + \dfrac{s}{\kappa} \int_0^s \|\partial_t \widetilde m + \calL \widetilde p\|^2_\Omega \dt&&\\
& + \dfrac{2}{\kappa} \int_0^s (\partial_t \widetilde m + \calL \widetilde p,\, \widetilde m)_\Omega \dt && (=: \mathfrak{I}_3)\\
& + \dfrac{7}{\gamma} \int_0^s \| \partial_\Normal \widetilde p \|^2_{H^{-\frac{1}{2}}(\partial \Omega)}\\ 
& + \dfrac{2}{\kappa} \| \widetilde m(0) \|^2_\Omega\dt.
\end{alignedat}
\end{equation}

We aim at bounding $\mathfrak{I}_1$, $\mathfrak{I}_2$ and $\mathfrak{I}_3$ from below. For
the first one, we mimic the argument in the proof of
\cite[Lemma~2.2]{Brenner.Sung:92}.  By
\cite[Lemma~2.1]{Brenner.Sung:92}, there is $v' \in L^2(\bbE)$ such
that 
 \begin{equation}
\label{E:infsup-strong-rightinverse}
\calD v' = \calD \widetilde u \quad \text{in $L^2(\bbP)$}
\qquad  \text{and} \qquad 
\Norm{D^2 v'}_\Domain \lesssim \Norm{\calD^* \calD \widetilde u}_\Domain
\end{equation}
with the hidden constant depending only on $\Domain$. The identity
$\calE = -\mu (\Delta + \nabla \Div )$ in $\bbE$ and simple
manipulations reveal 
\begin{equation*}
\begin{alignedat}{2}
-\mu \Delta ( \widetilde u - v' ) + \calD^*( \widetilde p_\Tot +
\mu\calD \widetilde u ) &= \widetilde f + \mu \Delta v' \quad &&\text{
  in $L^2([0, s], \,L^2(\Domain)^2)$}
 \\
\calD(\widetilde u - v') & = 0 &&\text{ in $L^2([0,s],\, \bbP)$}.
\end{alignedat}
\end{equation*}
with $\widetilde f := \calE \widetilde u + \calD^* \widetilde
p_\Tot$. By the regularity theory for the Stokes equations
\cite[Theorem~2]{Kellogg.Osborn:76}, a triangle inequality and
\eqref{E:infsup-strong-rightinverse}, we have 
\begin{equation}\label{eq:StokesEst}
\int_0^s\Big( \mu^2\|D^2 \widetilde u\|^2_\Omega + \| \calD^* (\widetilde p_{\Tot} +\mu \calD \widetilde u)\|^2_\Omega \Big)\dt 
\lesssim 
\int_0^s\Big( \| \widetilde f \|^2_\Omega + \mu^2 \Norm{\calD^*\calD \widetilde u}_\Domain^2 \Big).
\end{equation}
Again, the hidden constant depends only on $\Domain$. For the second
term on the left-hand side, we use triangle and Young's inequalities 
\begin{equation*}
\begin{split}
\| \calD^* (\widetilde p_{\Tot} + \mu \calD \widetilde u)\|^2_\Omega &= 
\| \calD^* \widetilde p_\Tot\|^2_\Omega + \mu^2 \| \calD^* \calD
\widetilde u\|^2_\Omega + 2\mu(\calD^* \widetilde p_\Tot,\, \calD^*
\calD \widetilde u)_\Omega
\\
&\geq \| \calD^* \widetilde p_\Tot\|^2_\Omega + \mu(\mu + \lambda) \|
\calD^*\calD \widetilde u\|^2_\Omega - 2\alpha\mu(\calD^* \widetilde
p,\, \calD^*\calD \widetilde u)_\Omega
\\
&\hphantom{=}- \dfrac{\mu}{\lambda}\|\calD^* (\lambda \calD \widetilde u - \widetilde p_\Tot - \alpha \widetilde p)\|^2_\Omega.
\end{split}
\end{equation*}
We choose $K_1$ to be the hidden constant in~\eqref{eq:StokesEst} and assume that 
the constant $\Lambda$ in~\eqref{E:regularity-assumption-parameters} is large enough to satisfy $2K_1 - 1 \leq \Lambda$, so that we have $K_1\mu\le \frac{\mu+\lambda}2$. Then inserting the last estimate into the previous one yields 
\begin{equation*}
\begin{split}
\mathfrak{I}_1 = \frac{K_1}{\mu} \int_0^s\| \widetilde f \|^2_\Omega&\geq  \int_0^s  \Big( -2\alpha(\calD^* \calD
\widetilde u,\, \calD^* \widetilde p)_\Omega + \mu \|D^2 \widetilde
u\|^2_\Omega + \frac{\mu + \lambda}2 \|\calD^*\calD \widetilde u
\|^2_\Omega\Big)\dt
\\
&+ \int_0^s\Big(  \dfrac{1}{\mu}\| \calD^* \widetilde
p_{\Tot}\|^2_\Omega - \dfrac{1}{\lambda}\|\calD^* (\lambda \calD
\widetilde u - \widetilde p_\Tot - \alpha \widetilde p)\|^2_\Omega
\Big)\dt.  
\end{split}
\end{equation*}

We estimate $\mathfrak I_2$ from below by Young's inequality 
\begin{equation*}
\begin{split}
\mathfrak I_2 \geq &\int_0^s \left( -\frac{\lambda}{4}\Norm{\calD^* \calD \widetilde u}^2_\Domain - \frac{1}{2\mu} \Norm{\calD^* \widetilde p_\Tot}^2_\Domain + \frac{\alpha^2 K_2}{\lambda}\Big( 1 - K_2 - \frac{K_2 \mu}{2\lambda}  \Big) \Norm{\calD^* \widetilde p}^2_\Domain \right)\dt\\
\geq & \int_0^s \left( -\frac{\lambda}{4}\Norm{\calD^* \calD \widetilde u}^2_\Domain - \frac{1}{2\mu} \Norm{\calD^* \widetilde p_\Tot}^2_\Domain + \frac{3\alpha^2 K_2}{4\lambda} \Norm{\calD^* \widetilde p}^2_\Domain \right)\dt.
\end{split}
\end{equation*}
The second lower bound holds true upon assuming that $K_2$ is small enough so that $K_2 (1 + \frac{1}{2\Lambda})\leq \frac{1}{4} $.

Regarding $\mathfrak{I}_3$, the integration by parts rule
\cite[Lemma~64.40]{Ern.Guermond:21c} and the identities
\eqref{E:abstract-operator-laplacian-gradient} and
\eqref{E:normal-derivative} reveal 
\begin{equation*}
\begin{split}
\mathfrak{I}_3 = \dfrac{1}{\kappa} \| \widetilde m(s) \|^2_\Omega -
\dfrac{1}{\kappa} \| \widetilde m(0) \|^2_\Omega + 2 \int_0^s \Big(
(\calD^* \widetilde p,\, \calD^*\widetilde m)_\Omega - \left\langle
  \partial_\Normal \widetilde p,\, \widetilde m \right\rangle_{1/2}
\Big)\dt. 
\end{split}
\end{equation*}
We bound the last term on the right-hand side, using \eqref{E:gamma} with $\overline{\bbP}=\bbD$,
\eqref{E:norm-1/2} and various Young's inequalities, by
\begin{equation*}
\begin{split}
\left\langle \partial_\Normal \widetilde p,\, \widetilde m \right\rangle_{1/2}
&\leq 
\dfrac{\sigma}{2} \Norm{\calD^* \widetilde p}_\Domain^2
+
\dfrac{\mu + \lambda}{6} \Norm{\calD^* \calD \widetilde u}_\Domain^2
\\
&+
\dfrac{3}{\gamma} \Norm{\partial_\Normal \widetilde p}_{H^{-\frac{1}{2}}(\partial \Domain)}^2 
+ 
\dfrac{\gamma}{4} 
 \| \calD^*(\alpha \calD \widetilde u + \sigma \widetilde p - \widetilde m) \|_\Omega^2.
\end{split}
\end{equation*}
Another application of Young's inequality and \eqref{E:gamma} yield
\begin{equation*}
\begin{split}
&(\calD^* \widetilde p,\, \calD^* \widetilde m)_\Domain \geq\\ 
& \qquad \qquad \Big( \sigma - \frac{\alpha^2 K_2}{4\lambda}\Big) \Norm{\calD^* \widetilde p}^2_\Domain - \frac{\gamma}{K_2} \Norm{\calD^*(\alpha \calD \widetilde u + \sigma \widetilde p - \widetilde m)}^2_\Domain + \alpha (\calD^* \calD \widetilde u, \calD^* \widetilde p)_\Domain.
\end{split}
\end{equation*}
Using these bounds into the previous identity and assuming $K_2 \leq 4$, it results
\begin{equation*}
\begin{split}
\mathfrak{I}_3 &\geq \dfrac{1}{\kappa} \| \widetilde m(s) \|^2_\Omega
- \dfrac{1}{\kappa} \| \widetilde m(0) \|^2_\Omega\\
&\hphantom{\geq} +\int_0^s \Big(
\big(\sigma -\frac{\alpha^2 K_2}{2\lambda}) \|\calD^* \widetilde p \|^2_\Domain  + 2\alpha (\calD^* \calD\widetilde u,\,
\calD^*\widetilde p)_\Omega\Big)\dt
\\
&\hphantom{\geq} -\int_0^s \Big(\dfrac{\mu + \lambda}{3} \|
\calD^*\calD\widetilde u \|^2_\Omega  + \dfrac{6}{\gamma}\| \partial_\Normal \widetilde p
\|^2_{H^{-\frac{1}{2}}(\partial \Omega)} + \frac{4\gamma}{K_2}\| \calD^*(\alpha \calD \widetilde u + \sigma \widetilde p - \widetilde m) \|_\Omega^2\Big). 
\end{split}
\end{equation*}

We combine the lower bounds for $\mathfrak{I}_1$, $\mathfrak{I}_2$ and $\mathfrak{I}_3$ with \eqref{E:infsup-strong-basic-estimate}. Recalling also the definition of $\gamma$ in \eqref{E:gamma}, we obtain
\begin{equation*}
b(\widetilde x_1, x_{2,s}) 
\gtrsim
\Normtr{\widetilde x_1}^2_1 
+ 
\dfrac{1}{\kappa}\Norm{\widetilde m(s)}^2_\Domain 
+ 
\int_0^s \Big(\lambda \| \calD^*\calD \widetilde u \|^2_\Omega 
+ 
\frac{1}{\gamma} \| \calD^* \widetilde p\|^2_\Omega \Big)\dt .
\end{equation*}
This establishes
\eqref{E:inf-sup-strong} upon taking the test function $x_2 =
x_{2,\overline s} + x_{2,T}$, where $\overline s$ is such that
$\|\widetilde m(\overline s)\|_\Omega = \| \widetilde m \|_{L^\infty(\overline\bbP)} $. Indeed, for all $s \in [0, T]$, the norm of $x_{2,s}$ is bounded by

\begin{equation*}
\begin{split}
\Normtr{x_{2,s}}_2^2 
&\lesssim 
\Normtr{\widetilde x_1}_1^2 + \int_0^T \Big( T^{-1} \| \widetilde m \|^2_\Omega + \lambda \Norm{\calD^* \calD \widetilde u}^2_\Domain\Big)\\
&\lesssim
\Normtr{\widetilde x_1}_1^2 + \| \widetilde m \|^2_{L^\infty(\overline{\bbP})} + \lambda\int_0^T \Norm{\calD^* \calD \widetilde u}^2_\Domain
\end{split}
\end{equation*}
and the hidden constant does not depend on $s$. The proof is identical to the corresponding one at the end of Section~\ref{SS:InfSup}. The combination of this estimate with the previous lower bound establishes \eqref{E:inf-sup-strong}.

\subsection{Nondegeneracy}
\label{SS:Nondegeneracy-strong}
This section establishes the nondegeneracy of the bilinear form $b$ in \eqref{E:BiotProblem-abstract-form-strong}. For this purpose, assume
\begin{equation}
\label{E:nondegeneracy-strong}
b(\widetilde x_1, x_2) = 0
\end{equation}
for all $\widetilde x_1 \in \overline \bbX_1$ and for some $x_2 = (v, q_\Tot, q, n, n_\partial, n_0) \in \bbX_2$. We proceed in analogy with Section~\ref{SS:Nondegeneracy} in order to verify $x_2 = 0$. Therefore, we mention only the main aspects of the argument.

Taking $\widetilde x_1 = (0, \widetilde p_\Tot, 0, 0)$, the identity \eqref{E:nondegeneracy-strong} implies
\begin{equation}
\label{E:nondegeneracy-strong-identity0}
\int_0^T (\calD^* \widetilde p_\Tot,\, v)_\Domain\dt
=
\int_{0}^T (\calD^* \widetilde p_\Tot,\, \calD^*q_\Tot)_\Domain
\end{equation}
for all $\widetilde p_\Tot \in L^2(\bbP)$.

Taking $\widetilde x_1 = (\widetilde u, 0,0,0)$, the identity \eqref{E:nondegeneracy-strong}, the inclusion $\calD \widetilde u \in L^2(\bbP)$ and \eqref{E:nondegeneracy-strong-identity0} imply
\begin{equation*}
\int_0^T (\calQ \widetilde u,\, v)_\Domain = -\int_0^T\alpha(\calD^*\calD \widetilde u,\, \calD^* q)_\Domain \dt
\end{equation*}
for all $\widetilde u \in L^2(\bbE)$, with $\calQ = \calE + \lambda
\calD^* \calD$ the operator involved in the displacement formulation
of the linear elasticity equations. According to
\cite[Lemma~2.2]{Brenner.Sung:92}, $\calQ$ is a one-to-one mapping
from $\bbE$ to $L^2(\Domain)^2$. Thus, the adjoint $\calQ^\star:
L^2(\Domain)^2 \to \bbE^*$ is invertible. Here `$\star$' denotes the
ajoint with respect to the $L^2(\Domain)$-scalar product, while `$*$'
indicates the one with respect to the duality $\left\langle \cdot,
  \cdot\right\rangle_\bbU$. Thus, 
\begin{equation}
\label{E:nondegeneracy-strong-identity1}
v = -\alpha \calQ^{-\star} (\calD^*\calD)^\star \calD^* q \qquad \text{ in $L^2(L^2(\Domain)^2)$}.
\end{equation}

Taking $\widetilde x_1 = (0, 0, \widetilde p, 0)$, the identities \eqref{E:abstract-operator-laplacian-gradient}, \eqref{E:normal-derivative} and \eqref{E:nondegeneracy-strong} imply
\begin{equation*}
\int_0^T \Big( (\calL \widetilde p,\, \kappa n -\alpha q_\Tot + \sigma
q)_\Domain  + \kappa \left\langle \partial_\Normal \widetilde p,\,
  n_\partial -\alpha q_\Tot + \sigma q\right\rangle_{1/2} \Big)\dt = 0 
\end{equation*}
for all $\widetilde p \in L^2(\bbL)$. We consider, in particular, the solution $\widetilde p$ of the Neumann problem 
\begin{equation*}
\begin{alignedat}{2}
\calL \widetilde p &= \kappa n - \alpha q_\Tot + \sigma q \; &&\text{ in $L^2(\overline{\bbP})$}\\
\partial_\Normal \widetilde p &= \mathcal{R}_\partial (n_\partial - \alpha q_\Tot + \sigma q) \; &&\text{ in $L^2(H^{-\frac{1}{2}}(\partial \Domain))$}
\end{alignedat}
\end{equation*}
where $\mathcal{R}_\partial: H^{\frac{1}{2}}(\partial \Domain) \to H^{-\frac{1}{2}}(\partial \Domain)$ is the Riesz isometry. We infer $n \in L^2(\bbP)$ with
\begin{equation}
\label{E:nondegeneracy-strong-identity2}
\kappa n = \alpha q_\Tot - \sigma q \quad \text{ in $L^2(\bbP)$}
\qquad \text{and} \qquad
\kappa n_{|\partial \Domain} = n_\partial \quad \text{ in $L^2(H^{\frac{1}{2}}(\Domain))$}.
\end{equation}

Finally, taking $\widetilde x_1 = (0, 0, 0, \widetilde m)$, the identities \eqref{E:abstract-operator-laplacian-gradient} and \eqref{E:nondegeneracy-strong} imply
\begin{equation*}
\int_{0}^T\Big( -\dfrac{1}{\kappa}\left\langle \widetilde m,\, \calL q\right\rangle _\bbP + (  \partial_t \widetilde m,\, n )_\Omega \Big)\dt + (\widetilde m(0),\, n_0)_\Domain = 0
\end{equation*}
for all $\widetilde m \in L^2(\bbP) \cap H^1(\overline \bbP)$. Thus, arguing as in the derivation of \eqref{E:nondegeneracy-identity3}-\eqref{E:nondegeneracy-identity4}, we infer $n \in H^1(\bbP^*)$ with 
\begin{equation}
\label{E:nondegeneracy-strong-identity3}
\partial_t n = -\dfrac{1}{\kappa}\calL q \quad \text{ in $L^2(\bbP^*)$}
\qquad \text{and} \qquad
n(0) = n_0, \; n(T) = 0 \quad \text{ in $\overline \bbP$}.
\end{equation}

We combine \eqref{E:nondegeneracy-strong-identity2} with
\eqref{E:nondegeneracy-strong-identity3}, and exploit also
\eqref{E:abstract-operator-laplacian-gradient},
\eqref{E:nondegeneracy-strong-identity0} as well as 
\eqref{E:nondegeneracy-strong-identity1} 
\begin{equation*}
%\label{E:nondegeneracy-strong-almost-concluded}
\begin{split}
-\dfrac{\kappa}{2} \Norm{n_0}^2_\Domain 
&= \kappa\int_0^T \left\langle \partial_t n,\, n\right\rangle_\bbP \dt 
= \dfrac{1}{\kappa}\int_0^T \left\langle \calL q,\, \sigma q - \alpha q_\Tot\right\rangle_\bbP\dt\\
&= \int_0^T \Big( \sigma\Norm{	\calD^*q}^2_\Domain + \alpha^2 (\calD^*\calD \calQ^{-1}\calD^*q,\, \calD^* q)_\Domain\Big)\dt . 
\end{split} 
\end{equation*}

We deal with the second term on the right-hand side by following once
again the proof of \cite[Lemma~2.2]{Brenner.Sung:92}. Recalling $\calQ
= -\mu\Delta + (\mu + \lambda) \calD^* \calD$ in $\bbE$, it follows
that 
\begin{equation*}
%\label{E:nondegeneracy-strong-almost-concluded2}
\int_0^T(\calD^*\calD \calQ^{-1}\calD^*q,\, \calD^* q)_\Domain\dt = \dfrac{1}{\mu + \lambda} \int_0^T\Big( \Norm{\calD^*q}^2_\Domain + \mu(\Delta u',\, \calD^*q)_\Domain \Big)\dt
\end{equation*}
with $u' := \calQ^{-1}\calD^* q \in L^2(\bbE)$
and consequently we have that
\begin{align}\label{E:nondegeneracy-strong-almost-concluded3}
    0=\dfrac{\kappa}{2} \Norm{n_0}^2_\Domain +\int_0^T \Big( \sigma\Norm{	\calD^*q}^2_\Domain + \dfrac{\alpha^2}{\mu + \lambda}\Norm{\calD^*q}^2_\Domain+ \dfrac{\alpha^2}{\mu + \lambda} \mu  (\Delta u',\, \calD^*q)_\Domain\Big)\dt. 
\end{align}    

By \cite[Lemma~2.1]{Brenner.Sung:92}, there is $v' \in L^2(\bbE)$ such that 
\begin{equation*}
\calD v' = \calD u' \quad \text{in $L^2(\bbP)$}
\qquad  \text{and} \qquad 
\Norm{D^2 v'}_\Domain \lesssim \Norm{\calD^* \calD u'}_\Domain
\end{equation*}
with the hidden constant depending only on $\Domain$. Then, the expression of $\calQ$ recalled above and simple manipulations reveal
\begin{equation*}
\begin{alignedat}{2}
-\mu \Delta ( u' - v' ) + \calD^*( (\mu +\lambda) \calD u' - q ) &= \mu \Delta v' \quad &&\text{ in $L^2(L^2(\Domain)^2)$}\\
\calD(u' - v') & = 0 &&\text{ in $L^2(\bbP)$}.
\end{alignedat}
\end{equation*}
By the regularity theory for the Stokes equations \cite[Theorem~2]{Kellogg.Osborn:76} and a triangle inequality, we have
\begin{equation*}
\begin{split}
\mu\Norm{D^2 u'}_\Domain + \Norm{\calD^*((\mu +\lambda) \calD u' - q )}_\Domain
&\lesssim 
\mu \Norm{\calD^* \calD u'}_\Domain
\\
&\leq \dfrac{\mu}{\mu + \lambda} \Big(  \Norm{\calD^*q} +
\Norm{\calD^*((\mu + \lambda) \calD u' - q )}_\Domain \Big). 
\end{split}
\end{equation*}
Again, the hidden constant depends only on $\Domain$.
  By virtue of~\eqref{E:regularity-assumption-parameters}, this actually implies
  \(\mu\|\Delta u'\|_\Omega\lesssim \frac{\mu}{\mu+\lambda} \|\calD^*
  q\|_\Omega\). 
Using a Cauchy-Schwartz inequality in~\eqref{E:nondegeneracy-strong-almost-concluded3}, this bound, in combination 
with  assumption~\eqref{E:regularity-assumption-parameters}, yields
\begin{equation*}
\begin{split}
\Big(\sigma + \dfrac{C\alpha^2}{\mu + \lambda}\Big)\int_0^T \Norm{	\calD^*q}^2_\Domain\dt + \dfrac{\kappa}{2} \Norm{n_0}^2_\Domain \le 0 
\end{split} 
\end{equation*}
for some positive constant $C > 0$. We conclude $x_2 = 0$ in \eqref{E:nondegeneracy-strong} by this
identity and
\eqref{E:nondegeneracy-strong-identity1}-\eqref{E:nondegeneracy-strong-identity3}.

\section{Conclusions and outlook}
\label{S:conclusions-outlook}

We have proposed a new approach and a corresponding
setting for the analysis of the quasi-static Biot's equations in
poroelasticity. In passing, we have relaxed the regularity assumptions
on the data formulated in previous references. The results here are
instrumental and tailored to our main goals, that are the design and
the analysis of discretizations enjoying accurate and robust error
bounds. To this end, we propose in \cite{Kreuzer.Zanotti:22+} a class
of discretizations inspired by the four-field formulation
\eqref{E:BiotProblem-equations-fourfield} and prove its stability by mimicking
the technique introduced in Section~\ref{S:WellPosedness}. The
stability estimate in Theorem~\ref{T:well-posedness} providesassumptionsassumptions a possible starting point
for the a posteriori error analysis. The regularity result in
Section~\ref{S:Regularity} is instrumental to the a priori error
analysis in \cite{Kreuzer.Zanotti:22+}, as it establishes the
regularity that is needed to infer first-order convergence in
space. Both the a posteriori analysis and the derivation of further
regularity results, e.g. relaxing the assumptions
\eqref{E:regularity-assumptions} or addressing the regularity in time,
may be the subject of future investigation.  

\subsection*{Acknowledgment} We are grateful to the reviewers for contributing to the revision process with many insightful comments.

\subsection*{Funding} Christian Kreuzer acknowledges funding by the Deutsche
For\-schungs\-ge\-mein\-schaft (DFG, German Research Foundation) -- 321270008. Pietro Zanotti was supported by the PRIN 2022 PNRR
project “Uncertainty Quantification of coupled models for water flow and contaminant transport”
(No. P2022LXLYY), financed by the European Union—Next Generation EU and by the GNCS-INdAM project CUP\_E53C23001670001.

\begin{figure}[htp]
	\includegraphics[width=0.9\hsize]{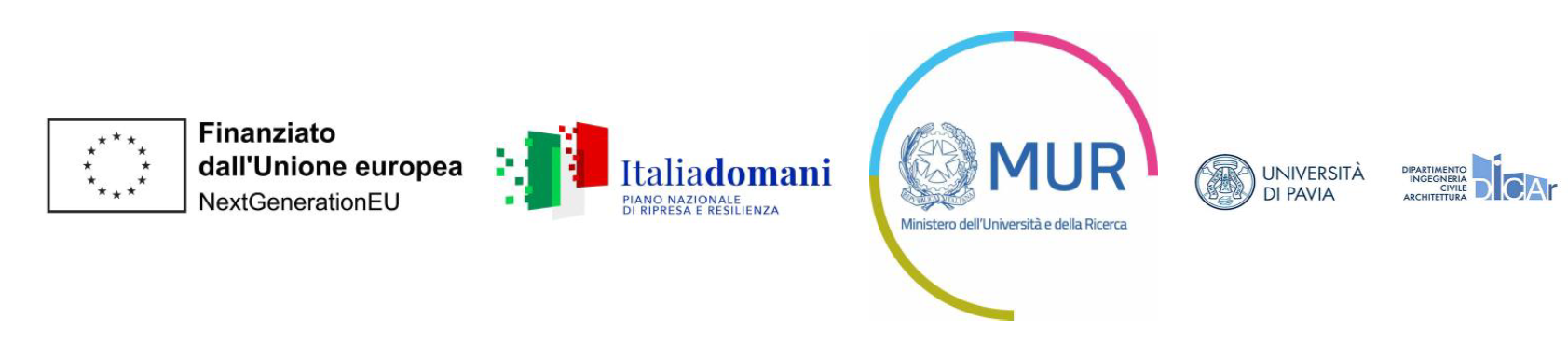}
\end{figure}

\end{document}